\documentclass[12pt,reqno]{amsart}
\def\today{\number\day\space\ifcase\month\or
  January\or February\or March\or April\or May\or June\or
  July\or August\or September\or October\or November\or December\fi
  \space\number\year}
\parskip=\medskipamount

\title[Paths to Understanding {\Large $\rho_{B}$}]{Paths to Understanding Birational
Rowmotion\\ on Products of Two Chains}
\author[G. Musiker]{Gregg Musiker}
\address{University of Minnesota\\ Minneapolis, MN}
\email{musiker@math.umn.edu}

\author[T. Roby]{Tom Roby}
\address{University of Connecticut\\Storrs, CT 06269-1009}
\email{tom.roby@uconn.edu}

\date{\today}

\usepackage{microtype}

\usepackage{amsfonts}
\usepackage{amsmath}
\usepackage{amssymb}
\usepackage{amsthm}

\usepackage{mathtools}

\usepackage{animate}

\usepackage{array}
\newcolumntype{L}[1]{>{\raggedright\let\newline\\\arraybackslash\hspace{0pt}}p{#1}}

\usepackage{longtable}

\usepackage{bm}

\usepackage{bbm}
\newcommand{\II}{\mathbbm 1}
\newcommand{\IIb}{\hat{\mathbbm 1}}

\usepackage[justification=centering,labelfont=bf]{caption}

\usepackage[dvipsnames]{xcolor}

\newcommand{\cred}[1]{{\color{red}#1}}
\newcommand{\cblu}[1]{{\color{blue}#1}}

\newcommand{\cgrn}[1]{{\color{OliveGreen}#1}}

\usepackage{setspace}

\usepackage{etex}

\usepackage{enumitem}

\usepackage{fancyhdr}
\usepackage[dvipdfmx]{graphicx} 
\usepackage{bmpsize}

\usepackage{listings}

\usepackage{manfnt}

\usepackage{scalerel}
\newcommand{\bigsig}{{\sum}}
\DeclareMathOperator*{\sumpar}{\bigsig^{\ \mathclap{\|}\mathclap{\rule[1.3pt]{6pt}{0.4pt}}}\:}
\def\parallelsum{\:\ \mathclap{\|}\mathclap{-}\ \:}
\def\parallel{\parallelsum} 

\usepackage{nicefrac}

\usepackage{textpos}

\usepackage{tikz}
\usetikzlibrary{matrix,backgrounds,positioning}
\newcommand*\ncirc[1]{\tikz[baseline=(char.base)]{
            \node[shape=circle,draw,inner sep=1.2pt] (char) {#1};}} 
\def\ncirc{\uline}

\usepackage[normalem]{ulem}

\usepackage{todonotes}
\definecolor{bananamania}{rgb}{0.98, 0.91, 0.71}
\definecolor{apricot}{rgb}{0.98, 0.81, 0.69}

\usepackage{url}
\usepackage[colorlinks=true,citecolor=black,linkcolor=black,urlcolor=blue,
 pdfauthor={Tom Roby}]{hyperref}

\newcommand{\arxiv}[1]{\href{http://arxiv.org/abs/#1}{\texttt{arXiv:#1}}}

\usepackage[ps,dvips,color,all]{xy}
\UseCrayolaColors
\usepackage{wasysym}
\def\lpreg{\hexagon}

\def\bc#1#2{\left(\kern -2pt{#1\atop #2} \kern -2pt\right)}
\def\bfit#1{{\textit{\textbf{#1}}}}

\def\bull{\noindent $\bullet$\kern 4pt}

\def\multiset#1#2{\ensuremath{\left(\kern-.3em\left(\genfrac{}{}{0pt}{}{#1}{#2}\right)\kern-.3em\right)}}

\newcommand{\ds}{\displaystyle}

\renewcommand{\hat}{\widehat}

\newtheorem{theorem}{Theorem}[section]

\newtheorem{claim}[theorem]{Claim}
\newtheorem{corollary}[theorem]{Corollary}

\newtheorem{definition}[theorem]{Definition}
\newtheorem{lemma}[theorem]{Lemma}

\theoremstyle{definition}
\newtheorem{eg}[theorem]{Example}
\newtheorem{defn}[theorem]{Definition}

\newtheorem{rem}[theorem]{Remark}

\catcode`@=11
\def\m@th{\mathsurround\z@}
\def\cases#1{\left\{\,\vcenter{\normalbaselines\m@th
    \ialign{$##\hfil$&\quad##\hfil\crcr#1\crcr}}\right.}
\def\matrix#1{\null\,\vcenter{\normalbaselines\m@th
    \ialign{\hfil$##$\hfil&&\quad\hfil$##$\hfil\crcr
      \mathstrut\crcr\noalign{\kern-\baselineskip}
      #1\crcr\mathstrut\crcr\noalign{\kern-\baselineskip}}}\,}

\def\hang{\hangindent 24pt}
\def\d@nger{\medbreak\begingroup\clubpenalty=10000
  \def\par{\endgraf\endgroup\medbreak} %
  \noindent\hang\hangafter=-2
  \hbox to0pt{\hskip-\hangindent\dbend\hfill}}
\outer\def\danger{\d@nger}

\catcode`@=12

\newcommand{\KK}{\mathbb K}

\newcommand{\NN}{\mathbb N}

\newcommand{\calB}{\mathcal{B}}

\newcommand{\calL}{\mathcal{L}}

\newcommand{\calO}{\mathcal{O}}

\newcommand{\calR}{\mathcal{R}}
\newcommand{\calS}{\mathcal{S}}

\newcommand{\frakL}{\mathfrak{L}}

\newcommand{\bfr}{r}
\newcommand{\bfs}{s}

\def\centerline#1{\hbox to \hsize{\hfill #1 \hfill}}

\newcommand{\whP}{{\widehat{P}}}
\newcommand{\rowmotion}{\rho}
\newcommand{\iscovered}{\lessdot}


\begin{document}

\maketitle

\begin{abstract}

Birational rowmotion is an action on the space of assignments of rational functions to the elements of a finite partially-ordered set (poset). It is lifted from the well-studied rowmotion map on order ideals (equivariantly on antichains) of a poset  $P$, which when iterated on special posets, has unexpectedly nice properties in terms of periodicity, cyclic sieving, and homomesy (statistics whose averages over each orbit are constant). In this context, rowmotion appears to be related to Auslander-Reiten translation on certain quivers, and birational rowmotion to $Y$-systems of type $A_m \times A_n$ described in Zamolodchikov periodicity.

We give a formula in terms of families of non-intersecting lattice paths for iterated actions of the birational rowmotion map on a product of two chains. This allows us to give a much simpler direct proof of the key fact that the period of this map on a product of chains of lengths $r$ and $s$ is $r+s+2$ (first proved by D.~Grinberg and the second author), as well as the first proof of the birational analogue of homomesy along files for such posets.

\end{abstract}

\section{Introduction}\label{sec:intro} 

The \textit{rowmotion} map $\rho$, defined on the set $J(P)$ of order ideals (equivariantly
on antichains) of a poset $P$, has been thoroughly studied by a number of combinatorialists
and representation theorists.  When iterated on special posets, particularly root posets and
(co)minuscule posets associated with representations of finite-dimensional Lie algebras,
$\rho$ has unexpected nice properties in terms of periodicity, cyclic sieving, and
homomesy~\cite{AST11,brouwer-schrijver, CF95,Pan09,PR13,RuSh12,RuWa15+,SW12,ThWi17}.
Excellent summaries of the history of this map and further references are available in the
introductions of Striker-Williams~\cite{SW12} and Thomas-Williams~\cite{ThWi17}.

Armstrong, Stump, and Thomas~\cite{AST11} proved a conjecture of
Panyushev~\cite[Conj. 2.1(iii)]{Pan09} that under the action of rowmotion on antichains of
root posets, for any orbit $\calO$, the value $\frac{1}{\#\calO}\sum_{A\in \calO} \#A$ is a
constant, independent of the choice of $\calO$.  This was one of the first explicit
statements of a type later isolated by Propp and the second author as the much more
widespread \textbf{homomesy phenomenon}~\cite{PR13} (see Definition~\ref{def:homomesy}).
(In particular, the antichain-cardinality is homomesic with respect to rowmotion on
antichains of root posets.)   Propp and Roby's result that cardinality is a homomesic
statistic for rowmotion acting on $J([a]\times [b])$ was generalized by Rush and Wang to
all minuscule posets~\cite{RuWa15+}.  In the cominuscule (equivalently minuscule) context, there is
a still mysterious connection between rowmotion on 
certain posets and Auslander-Reiten translation on related quivers~\cite{Yil17}.

Considering an order ideal $I\in J(P)$ as an order-preserving map $I:P\rightarrow \{0,1 \}$
leads naturally to a generalization of $\rho$ to a piecewise-linear action $\rho_{PL}$ on
the \textit{order polytope}~\cite{Stan86} of $P$, i.e., $\{f:P\rightarrow [0,1]: f 
\text{ is order preserving}\}$.   This is then detropicalized (a.k.a.\ ``geometricized'') to
a birational map $\rho_B$, as detailed in~\cite{EiPr13,EiPr14}, following in the footsteps
of Kirillov and Berenstein~\cite{KiBe95}.  A key aspect of this is the insight
of Cameron and Fon-der-Flaass that $\rho$ can be decomposed as a product of \emph{toggles},
i.e., involutions defined for each element of $P$; thus, to generalize $\rho$ to other maps, it suffices
to generalize the definition of toggles.  None of this background is logically necessary for the current
paper, but it serves as motivation for why the birational rowmotion map $\rho_{B}$ is of interest.  

At the birational level, $\rho_{B}$ is a map on the set of assignments of rational functions
to the elements of the poset (with some genericity assumptions or domain restrictions to avoid
dividing by zero).  Theorems proven at the birational
level generally imply their corresponding theorems at the piecewise-linear level, and then at the
combinatorial level, but not vice-versa.  For example, the only proof available as of this
writing to show that piecewise-linear rowmotion is periodic uses the corresponding result for
birational rowmotion (Corollary~\ref{cor:periodicity}).

Periodicity of birational rowmotion was proved by Grinberg and Roby for a number of special
classes of posets, including for \emph{skeletal posets} (a generalization of graded
forests)~\cite{GrRo16,GRarX} and for triangles and rectangles~\cite{GrRo15,GRarX}, with the
latter being the fundamental and most challenging case.  
In this paper we give a formula in terms of families of non-intersecting lattice paths for iterated
actions of the \textit{birational rowmotion} map $\rho_{B}$ on a product of two chains.
This allows us to give a direct and significantly simpler proof that $\rho_{B}$ is periodic,
with the same period as ordinary (combinatorial) rowmotion
(Corollary~\ref{cor:periodicity}).   In this context, the \emph{homomesy phenomenon}
manifests 
itself as ``constant products over orbits'' since arithmetic means get replaced with
geometric means in the detropicalization process by which $\rho_{B}$ is defined.  
We apply our formula to prove two fundamental instances of homomesy for birational rowmotion on a
product of two chains: \emph{reciprocity} (Corollary~\ref{cor:reciprocity}) and \emph{file
homomesy} (Theorem~\ref{thm:homomesy}).  It
is expected that for the product of two chains, all ``natural'' homomesies for birational rowmotion can
be constructed as multiplicative combinations of these two~\cite[\S 10--11]{EiPr13}, in parallel
with the situation for the action of ordinary (combinatorial) rowmotion~\cite[\S4.1]{PR13}.

Many proofs of periodicity or homomesy in dynamical algebraic combinatorics involve
finding an equivariant bijection between rowmotion and an action that is easier to
understand, or at least already better understood.  At the combinatorial level, rowmotion
can be equivariantly and bijectively mapped, via the Stanley-Thomas
word, to bitstrings under cyclic rotation~\cite[\S 7]{PR13}.  For birational rowmotion, Grinberg and Roby
parameterize poset labelings by ratios of determinants, and then show periodicity 
and reciprocity via certain Pl\"ucker relations (overcoming a number of technical hurdles)~\cite{GrRo15}.  
By contrast, the methods of this paper involve working directly from our path formula, yielding
more explicit direct proofs of periodicity (Corollary~\ref{cor:periodicity}) and the reciprocity homomesy
(Corollary~\ref{cor:reciprocity}).  Additionally, our  
methods yield the first proof of a birational homomesy result along \emph{files} of our poset, namely that  
the product over all iterates  
of birational rowmotion over all elements of a given file is equal to $1$
(Theorem~\ref{thm:homomesy}).  This was first stated in Einstein-Propp~\cite[Thm.~9 and
remarks below Cor.~7]{EiPr13}, with some ideas of how one might construct a possible (more
indirect) proof.

The paper is organized as follows.  In Section~\ref{sec:defsMain}  we give basic
definitions, state our main 
result (Theorem~\ref{thm:main}, the lattice path formula for iterating birational
rowmotion), and present an extended illustrative example.  
We then state the main 
applications of our formula (periodicity and homomesy of $\rho_{B}$), deferring complicated arguments to the end of
Section~\ref{sec:proofs}.

In Section~\ref{sec:proofs} 
we prove our formula by way of some colorful combinatorial bijections for pairs of families
of non-intersecting lattice paths. Similar bijections have appeared earlier in the
literature, notably the paper of Fulmek and Kleber~\cite{FuKl01}.  (We are grateful to
Christian Krattenthaler for pointing us to their work.)  This section ends with a
proof of file homomesy,  Theorem~\ref{thm:homomesy}, using the aforementioned lattice path
formula, a telescoping sequence of cancellations, and an equality proven via a
double-counting argument. 
In Section~4 we conclude with connections to other work and directions for further research.

\subsection{Acknowledgements}\label{ss:acknowl}

The initial impetus for this paper came when the authors were both at a workshop on
``Dynamical Algebraic Combinatorics'' hosted by the American Institute of Mathematics (AIM)
in March~2015. We are grateful for AIM's hospitality that week, as well as for helpful
conversations with Max Glick, Darij Grinberg, Christian Krattenthaler, James Propp, Pasha
Pylyavskyy, and Richard Stanley.  We especially thank Grinberg for his careful reading of our first draft.  
Computations of birational rowmotion in SageMath~\cite{sage} were of enormous assistance in helping us discover the main
formula.  The second author thanks the the Erwin Schr\"odinger Institute (Vienna) for its
hospitality during a workshop on \emph{Enumerative and Algorithmic Combinatorics}, at which
these results were presented.

\section{Definitions and main result}\label{sec:defsMain}

\subsection{Definition of birational rowmotion} \label{ss:def}

Birational rowmotion can be defined for any
labeling of the elements of a finite poset by elements of a field.  The original motivation for
considering this came from the work of Einstein and Propp~\cite{EiPr13, EiPr14} (following
work of Kirillov-Berenstein~\cite{KiBe95}), which explained
how to lift the notion of \textbf{toggles}:\ first from the combinatorial setting to the
piecewise-linear setting, and second from the piecewise-linear setting to the birational
setting via ``detropicalization''.  This allowed them to define piecewise-linear and
birational analogues of rowmotion, which they wished to study from the standpoint of
homomesy, whose traditional definition requires a periodic action.  So they were eager to
have a proof of periodicity, which was first supplied in~\cite{GrRo15}.  Another exposition
of these ideas and further background can be found 
in~\cite[\S4]{Rob16}.  For basic information and notation about posets, we direct the reader
to~\cite[Ch. 3]{Stan11}.  

\begin{definition}\label{def:bitoggle}
Let $P$ be any finite poset, and let $\whP$ be $P$ with an additional global maximum (denoted
$\hat 1 $) and an additional global minimum (denoted $\hat 0 $) adjoined.  
Let $\KK$ be any field, and $f\in \KK^{\whP}$ be any labeling of the elements of $\whP$ by
elements of $\KK$.  We define the \textbf{birational toggle} 
$T_{v}:\mathbb{K}^{\widehat{P}}\dashrightarrow\mathbb{K}^{\widehat{P}}$
at $v\in P$ by
\begin{equation}
\left(  T_{v}f\right)  \left(  y\right)  =\left\{
\begin{array}{cc}
f\left(  y\right)  ,\ \ \ \ \ \ \ \ \ \ & \text{if }y\neq v;\\[8pt]
\dfrac{1}{f\left(  v\right)  }\cdot\dfrac{\sum\limits_{\substack{w\in
\widehat{P};\\w\lessdot v}}f\left(  w\right)  }{\sum\limits_{\substack{z\in
\widehat{P};\\z\gtrdot v}}\dfrac{1}{f\left(  z\right)  }}
,\ \ \ \ \ \ \ \ \ \ & \text{if }y=v
\nonumber
\end{array}
\right.
\end{equation}
for all $y \in \whP$.
Note that this rational map $T_{v}$ is
well-defined, because the right-hand side is
well-defined on a Zariski-dense open subset of $\mathbb{K}^{\widehat{P}}$.
Finally, we define \textbf{birational rowmotion} by $\rowmotion_{B}: = T_{v_{1}}T_{v_{2}}\dots
T_{v_{n}}: \KK^{\whP}\dashrightarrow \KK^{\whP}$, where $v_{1},v_{2},\dotsc ,v_{n}$ is any
linear extension of $P$. (``Toggling at each element of $P$ from top to bottom.'')   
\end{definition}
The toggle map $T_{v}$ changes only the label of the poset at $v$, and does this by 
(a) \emph{inverting} the label at $v$, and (b) multiplying by the \emph{sum} of the labels
at vertices \emph{covered by} $v$, and (c) multiplying by the \emph{parallel sum} of
the labels at vertices \emph{covering} $v$. 
It is lifted from a piecewise-linear toggle given by 
\[
f \mapsto  \left\{ \begin{array}{ll}
f(y) & \mbox{if $y \neq v$}, \\
\min_{z\,\cdot> v} f(z) + \max_{w <\cdot\,v} f(w) - f(v) & \mbox{if $y = v$},
\end{array} \right.
\]
Using the relation $\min (z_{i})=-\max (-z_{i})$, lifting $\max$ to $+$ forces us to lift
$\min$ to 
the (associative) parallel sum operation $\parallel$, defined by
$a\parallel b : = \frac{1}{\frac{1}{a}+\frac{1}{b}}$ (when $a,b\neq 0$ and $a\neq -b$).    

The main result of our paper is a formula in terms of families of non-intersecting lattice
paths for the $k$th iteration, $\rho_{B}^{k}$, of birational rowmotion on the product of two
chains. Accordingly, we will henceforth let $P$ denote this specific poset, i.e., the product of two chains. 
For our purposes, it is more convenient to coordinatize our poset $P$ as
$[0,r]\times [0,s]$ (where $[0,n] = \{0,1,2,\dots, n\}$), with minimal element $(0,0)$,
maximal element $(r,s)$ and covering relations:  
$(i,j) \lessdot (k,\ell)$ if and only if (1) $i=k$ and $\ell  = j +1$ or (2) $j=\ell$ and $k
= i+1$. The poset $P$ is clearly a \emph{graded} poset, where the \emph{rank} of $(i,j)$ is $i+j$.  
Orthogonally, for $k$ fixed, we call $F:=\{(i,j)\in P: j-i=k \}$ the $k$th 
\emph{file} of $P$.  

We then initially assign the generic label $x_{ij}$ (a.k.a.\ $x_{i,j}$) to the element $(i,j)$, and the
label 1 to the elements $\hat{0}$ and $\hat{1}$.  No essential generality is lost by
assigning 1 to the elements of $\whP -P$ (a ``reduced labeling'')~\cite[\S 4]{GrRo15} or
\cite[\S 4]{EiPr13}, but it simplifies our formulae and figures, which will generally just
display the labelings of $P$ itself, not of $\whP$.

\begin{eg}\label{eg:Hasse}
The Hasse diagram of $P = [0,2]\times [0,3]$ is shown on the left, with file $F=\{(2,3),
(1,2), (0,1) \}$ highlighted in \cred{red}.  The generic initial
labeling $f$ of $\whP$ is shown on the right.  

\[
\xymatrixrowsep{0.8pc}\xymatrixcolsep{0.18pc}\xymatrix{
&&& & \cred{(2,3)} \ar@{-}[rd] \ar@{-}[ld] & \\
&&& (2,2) \ar@{-}[rd] \ar@{-}[ld] & & (1,3) \ar@{-}[ld] \ar@{-}[rd]\\
P=&&(2,1) \ar@{-}[rd] \ar@{-}[ld] & & \cred{(1,2)} \ar@{-}[ld] \ar@{-}[rd] && (0,3) \ar@{-}[ld] & \\
&(2,0)  \ar@{-}[rd] && (1,1) \ar@{-}[rd] \ar@{-}[ld] & & (0,2) \ar@{-}[ld] && && && f=\\ 
&& (1,0)  \ar@{-}[rd] && \cred{(0,1)} \ar@{-}[ld] & &\\
&&& (0,0)
}
\xymatrixrowsep{0.8pc}\xymatrixcolsep{0.18pc}\xymatrix{
&&& 1 \ar@{-}[d] \\ 
&& & x_{23} \ar@{-}[rd] \ar@{-}[ld] & \\
&& x_{22} \ar@{-}[rd] \ar@{-}[ld] & & x_{13} \ar@{-}[ld] \ar@{-}[rd]\\
&x_{21} \ar@{-}[rd] \ar@{-}[ld] & & x_{12} \ar@{-}[ld] \ar@{-}[rd] && x_{03} \ar@{-}[ld] & \\
x_{20}  \ar@{-}[rd] && x_{11} \ar@{-}[rd] \ar@{-}[ld] & & x_{02} \ar@{-}[ld] && \\ 
& x_{10}  \ar@{-}[rd] && x_{01} \ar@{-}[ld] & &\\
&& x_{00} \ar@{-}[d]\\ 
&& 1
}
\]
\end{eg}

\begin{eg}\label{eg:2x2toggle}
Consider the $4$-element poset $P:= \left[  0,1\right]  \times\left[
0,1\right]  $, i.e., the product of two chains of length one, with the subscript-avoiding
labeling shown below.  
Then $f$ and the output of toggling $f$ at the top element $(1,1)$ of $P$ are
as follows.
\[
\xymatrixrowsep{0.9pc}\xymatrixcolsep{0.20pc}\xymatrix{
& &1\ar@{-}[d] & \\
& & z \ar@{-}[rd] \ar@{-}[ld] & \\
f = \ & x \ar@{-}[rd] & & y \ar@{-}[ld] \\
& & w \ar@{-}[d] \\
& & 1 & 
} 
\qquad  \qquad \raisebox{-55pt}{$\rightsquigarrow$} \qquad \qquad 
\xymatrixrowsep{0.9pc}\xymatrixcolsep{0.20pc}\xymatrix{
& &1\ar@{-}[d] & \\
& & {\color{red} \frac{\left(x+y\right)}{z}} \ar@{-}[rd] \ar@{-}[ld] & \\
T_{\left(  1,1\right)  } f = \ & x \ar@{-}[rd] & & y \ar@{-}[ld] \\
& & w \ar@{-}[d] & \\
& & 1 &
}
\]
Since the labels at $\hat{0}$ and $\hat{1}$ never vary, we suppress displaying them in all future examples
of birational rowmotion. (They are still involved in the computations.)   
Computing successively $T_{\left(  0,1\right)
}T_{\left(  1,1\right)  }f$, then
$T_{\left(  1,0\right)  }T_{\left(  0,1\right)  }T_{\left(  1,1\right)
}f$, and finally $\rho_{B}f=T_{\left(  0,0\right)  }T_{\left(  1,0\right)  }T_{\left(  0,1\right)
}T_{\left(  1,1\right)  }f$ gives: 
\[
\xymatrixrowsep{0.9pc}\xymatrixcolsep{0.20pc}\xymatrix{
& & \frac{\left(x+y\right)}{z} \ar@{-}[rd] \ar@{-}[ld] & \\
\ & x \ar@{-}[rd] & & {\color{red} \frac{\left(x+y\right)w}{yz}} \ar@{-}[ld] \\
& & w  & ,\\
}
\quad 
\xymatrixrowsep{0.9pc}\xymatrixcolsep{0.20pc}\xymatrix{
& & \frac{\left(x+y\right)}{z} \ar@{-}[rd] \ar@{-}[ld] & \\
& {\color{red} \frac{\left(x+y\right)w}{xz}} \ar@{-}[rd] & & \frac{\left(x+y\right)w}{yz} \ar@{-}[ld] \quad \text{ and}\\
& & w  & ,\\
}\quad 
\xymatrixrowsep{0.9pc}\xymatrixcolsep{0.20pc}\xymatrix{
& & \frac{\left(x+y\right)}{z} \ar@{-}[rd] \ar@{-}[ld] & \\
& \frac{\left(x+y\right)w}{xz} \ar@{-}[rd] & & \frac{\left(x+y\right)w}{yz} ~~~~~. \ar@{-}[ld] \\
& & {\color{red} \frac{1}{z}}  & \\
}
\]
\end{eg}

\begin{eg}\label{eg:2x2}
By repeating this
procedure (or just substituting the labels of $\rho_{B}f$ obtained as variables), we
can compute the iterated maps $\rho_{B}^{2}f$, $\rho_{B}^{3}f, \dots $ obtaining
\begin{align*}
&
\xymatrixrowsep{0.9pc}\xymatrixcolsep{0.20pc}\xymatrix{& & \frac{\left(x+y\right)}{z} \ar@{-}[rd] \ar@{-}[ld] & \\ 
 \rho_Bf = \ & \frac{\left(x+y\right)w}{xz} \ar@{-}[rd] & & \frac{\left(x+y\right)w}{yz} \ar@{-}[ld] \\ & & \frac{1}{z}  & ,}\ \ \ \ \ \ \ \ \ \ 
\xymatrixrowsep{0.9pc}\xymatrixcolsep{0.20pc}
\xymatrix{ & & \frac{\left(x+y\right)w}{xy} \ar@{-}[rd] \ar@{-}[ld] & \\   
\rho_{B}^2 f = \ & \frac{1}{y} \ar@{-}[rd] & & \frac{1}{x} \ar@{-}[ld] \\ & & \frac{z}{x+y}  & ,}\\
& \\
&
\xymatrixrowsep{0.9pc}\xymatrixcolsep{0.20pc}\xymatrix{ & &
\frac{1}{w} \ar@{-}[rd] \ar@{-}[ld] & \\ \rho_{B}^3 f = \ & \frac{yz}{\left(x+y\right)w}
\ar@{-}[rd] & & \frac{xz}{\left(x+y\right)w} \ar@{-}[ld] \\ & &
\frac{xy}{\left(x+y\right)w}& ,}\ \ \ \ \ \ \ \ \ \
\xymatrixrowsep{0.9pc}\xymatrixcolsep{0.20pc}\xymatrix{ & & z
\ar@{-}[rd] \ar@{-}[ld] & \\ \rho_{B}^4f = \ & x \ar@{-}[rd] & & y \ar@{-}[ld] \\ & & w  & .}
\end{align*}
Even this small example illustrates several interesting properties of this action.  
Notice that $\rho_B^{4}f=f$, which generalizes to $\rho_{B}^{r+s+2}f=f$ for $P=[0,r]\times
[0,s]$ (Corollary~\ref{cor:periodicity}). 
More subtly, as one iterates $\rho_{B}$, the labels at certain poset elements are reciprocals
of others occuring earlier at the antipodal position in the poset $P$.  For example\footnote{To avoid notation with double parentheses, we write $f(a,b)$ for $f(v)$ whenever $v=(a,b)$ in the following.},   
$\left(  \rho_Bf\right)  \left(  0,0  \right)  = 1/f(1,1)$,
$\left(  \rho_B^{2}f\right)  \left(  0,1  \right) = 1/f(1,0) $, $\left(
\rho_B^{2}f\right)  \left( 1,0 \right) = 1/f(0,1)  $, $\left(
\rho_B^{3}f\right)  \left( 1,1 \right) = 1/f(0,0)  $, and these induce further 
relations such as 
$\left(  \rho_B^{2}f\right)  \left(  0,0  \right) = 1/\left(  \rho_Bf\right)  \left( 1,1  \right)$.
This \textquotedblleft reciprocity\textquotedblright\ phenomenon turns out to
generalize to arbitrary rectangular posets (Corollary~\ref{cor:reciprocity}). 
\end{eg}

\begin{eg}\label{eg:2x2files}
We also note that the poset $P=[0,1]\times [0,1]$ has three files, namely $\{(1,0)\}$, $\{(0,0),(1,1)\}$, and $\{(0,1)\}$.
We observe the following identities, one per file, as we multiply over all iterates of
birational rowmotion the values of all the elements in a given file:
\[
\prod_{k=1}^{4}\rho_B^k(f)(1,0) = \frac{(x+y)w}{xz}~~ \frac{1}{y}~~ \frac{yz}{(x+y)w}~~ (x) = 1,
\]
\[
\prod_{k=1}^{4}\rho_B^k(f)(0,0)\rho_B^k(f)(1,1) =
\frac{1}{z} ~~ \frac{x+y}{z} ~~  \frac{z}{x+y}  ~~ \frac{(x+y)w}{xy}  ~~ \frac{xy}{(x+y)w}  ~~ \frac{1}{w} ~~  (w) ~~ (z) = 1,
\]
\[
\prod_{k=1}^{4}\rho_B^k(f)(0,1) = \frac{(x+y)w}{xz}~~ \frac{1}{y}~~ \frac{yz}{(x+y)w}~~ (x) = 1. 
\]
The fact that each of these products equals $1$ is the manifestation of homomesy along files
(of the poset of a product of two chains) at the birational level (Theorem~\ref{thm:homomesy}).

\end{eg}

\subsection{Our Main Result: A lattice path formula for birational rowmotion}

Here we state our main result, Theorem~\ref{thm:main}.  It gives a formula for any iteration
of $\rho_{B}$ on a product of two chains, as a 
ratio of polynomials in $A$-variables (simple fractions of the $x_{ij}$'s), where each
monomial corresponds to a family of non-intersecting lattice paths (NILPs). 
As a corollary, we give simpler
and more direct proofs that the period of this map on a product of chains of lengths
$r$ and $s$ is $r+s+2$ and that it satisfies the homomesies on display in the previous
examples.  

A simple change of variables in the initial labeling greatly facilitates our ability to write
the formula.  Let 
\begin{equation}\label{eq:Aij}
A_{ij} := \frac{\sum_{z\lessdot (i,j)} x_{z}}{x_{ij}}=\frac{x_{i,j-1} + x_{i-1,j}}{x_{ij}}, 
\end{equation}
where in particular $A_{i0} = \frac{x_{i-1,0}}{x_{i,0}}$, $A_{0j} = \frac{x_{0,j-1}}{x_{0,j}}$ 
and $A_{00} = \frac{1}{x_{00}}$ (working in
$\whP$). 

We define a \textbf{lattice path of length $\ell$} within $P=[0,r]\times [0,s]$ to be
a sequence $v_{1}, v_{2}, \dotsc , v_{\ell}$ of elements of $P$ such that each difference of
successive elements, $v_{i}-v_{i-1}$, is either $(1,0)$ or $(0,1)$ for each $2\leq i \leq \ell$.  We call
a collection of lattice paths \textbf{non-intersecting} if no two of them share a common
vertex.  We will frequently abbreviate \textbf{non-intersecting lattice paths} as
\textbf{NILPs}. 

\begin{defn}\label{def:phi}
Given a triple $(k,m,n)\in \NN^{3}$ (where $\NN$ denotes the nonnegative integers
$\{0,1,2,\dots \}$) with $k\leq \min \{r-m, s-n \}+1$, we define 
a polynomial $\bm{\varphi_k(m,n)}$ in terms of the $A_{ij}$'s as follows: 

1) Let $\bigvee_{(m,n)} :=\{(u,v): (u,v)\geq (m,n)\}$ be the \emph{principal order
filter at $(m,n)$} in $P$, which is isomorphic to $[0,r-m]\times [0,s-n]$.  Set
$\lpreg_{(m,n)}^{k} : = \{(u,v)\in \bigvee_{(m,n)}: 
m+n+k-1\leq u+v\leq r+s-k+1 \}$, the \emph{rank-selected
subposet} consisting of all elements in $\bigvee_{(m,n)}$ whose rank (within
$\bigvee_{(m,n)}$) is at least $(k-1)$ and whose corank is at least $(k-1)$.  

2) More specifically, let $s_1, s_2, \dots, s_k$ be the $k$ minimal elements and
$t_1, t_2, \dots, t_k$ be the $k$ maximal elements of $\lpreg_{(m,n)}^{k}$,
i.e.,  $s_{\ell } = (m+k-\ell ,n+\ell-1 )$ and $t_{\ell } = \left(r-\ell +1,s-k+\ell
\right)$ for $\ell  \in [k]$.  (When $k=0$, there are no 
$s_{\ell}$'s or $t_{\ell}$'s.)  Our condition that $k\leq \min \{r-m, s-n \}+1$ insures that
these points all lie within $\lpreg_{(m,n)}^{k}$.  

3) Let $S_k(m,n)$ be the set of families of NILPs in $\lpreg_{(m,n)}^{k}$ 
from $\{s_1,s_2,\dots, s_k\}$ to $\{t_1,t_2,\dots, t_k\}$.  We let $\calL = \{L_1,L_2,\dots, 
L_k \}\in S_{k}(m,n)$ denote such a family.

4) Define 
\begin{equation}\label{eq:phi}
\varphi_k(m,n) : = \sum_{ \calL  \in S_{k}(m,n)}
\hspace{2em}\prod_{\stackrel{(i,j) \in \lpreg_{(m,n)}^{k}}{(i,j) \not \in L_1 \cup L_2 \cup \dots \cup
L_k}} A_{ij}. 
\end{equation}

5) Finally, set $[\alpha ]_{+} : = \max \{ \alpha ,0 \}$ and let $\mu^{(a,b)}$ be the transformation
that takes a 
rational function in $\{A_{u,v}\}$ and simply shifts each index in 
each factor of each term: $A_{u,v}\mapsto A_{u-a,v-b}$.
\end{defn}

We are now ready to state our main result.

\begin{theorem}\label{thm:main} Fix $k\in [0, r+s+1]$, and let
$\rho_B^{k+1}(i,j)$ denote the rational function in $\KK[x_{u,v}]$ associated to the poset element $(i,j)$
after $(k+1)$ applications of the birational rowmotion map to the generic initial labeling
of $P = [0,r] \times [0,s]$.  
Set $M = [k-i]_+ + [k-j]_+$.  We obtain the following formula for $\rho_B^{k+1}(i,j)$: 

\bfit{(a)}  When $M\leq k$:

\begin{equation}\label{eq:main1}
\rho_B^{k+1}(i,j) = \mu^{([k-j]_+,
[k-i]_+)}\left(\frac{\varphi_{k-M}(i-k+M,j-k+M)}{\varphi_{k-M+1}(i-k+M,j-k+M)}\right)
\end{equation}

where $\varphi_t(v,w)$ and $\mu^{(a,b)}$ are as defined in 4) and 5) of Definition~\ref{def:phi}.  

\bfit{(b)} When $M \geq k$: 
\[
\rho_{B}^{k+1}(i,j) = 1/\rho_{B}^{k-i-j}(r-i, s-j)
\]
where $\rho_{B}^{k-i-j}(r-i,s-j)$ is well-defined by part (a). 
\end{theorem}
\begin{rem} \label{rem:st}
We note that in the above formulae we only ever use $\varphi_k(m,n)$'s such that the triple $(k,m,n)$ satisfies
the hypothesis of Definition~\ref{def:phi}.  In particular, in part (a)
we deduce $$0 \leq k -M < k-M+1 \leq \min \{r-i+k-M, s-j+k-M\} + 1$$ from the two inequalities $r-i\geq 0$, $s-j\geq 0$, which both 
follow from $(i,j) \in [0,r]\times [0,s]$.
\end{rem}
\begin{rem}
Note that our formulae in (a) and (b) agree when $M=k$, as we will see as part of
Claim~\ref{claim:Mk}.  Additionally, we see 
that the formula $\rho_{B}^{k-i-j}(r-i,s-j)$ satisfies the hypotheses for part (a) as
follows: First by letting $K = k-i-j-1$, $I = r-i$ and $J=s-j$, we see that the formula $\rho_{B}^{k-i-j}(r-i,s-j) = \rho_B^{K+1}(I,J)$ is well-defined by part (a) if $[K-I]_+ + [K-J]_+ \leq K$.  Second, we assume that $(K-I)$ and $(K-J)$ are both positive, because this inequality holds automatically if one or both of 
$(K-I)$ or $(K-J)$ are negative.  Then the only way the hypothesis would fail is if 
$(K-I) + (K-J) > K$, i.e., 
$$(k-i-j-1) - (r-i) + (k-i-j-1) - (s-j) = 2k - r - s - 2 - i - j > k - i - j -1.$$ But that implies that $k > r + s + 1$, contracting our assumption $k \in [0,r+s+1]$.
\end{rem}

Since on $P = [0,r]\times [0,s]$ we have $\rho_{B}^{r+s+2+d} = \rho_{B}^{d}$ by
periodicity (Corollary~\ref{cor:periodicity}), this gives a formula for \textbf{all}
iterations of the birational rowmotion map on $P$.

In the ``generic'' case where shifting $(i,j)\mapsto (i-k, j-k)$ (straight down by $2k$ ranks)
still gives a point in $P$, we get the following much simplified formula (which we discovered
first and then generalized to the main theorem).  
\begin{corollary}\label{cor:noShiftMain}
\begin{equation}\label{eq:noShift}
\text{For }k\leq \min \{ i,j\},\  \rho_B^{k+1}(i,j) = \frac{\varphi_{k}(i-k,j-k)}{\varphi_{k+1}(i-k,j-k)}.
\end{equation}
\end{corollary}

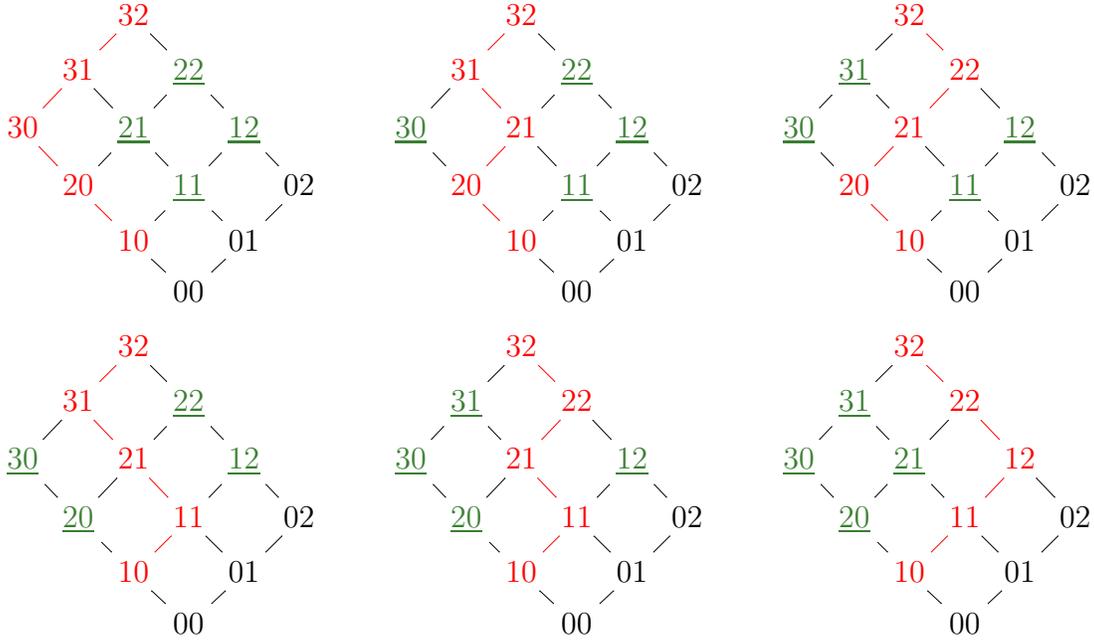
\begin{figure}[h]
\caption{The six lattice paths (shown in \cred{red}) used to compute $\varphi_{1}(1,0)$ in $[0,3]\times
[0,2]$. Corresponding $A$-variable subscripts are underlined in \cgrn{\ncirc {green}}.}
\label{fig:6LP}
\[
\xymatrixrowsep{0.4pc}\xymatrixcolsep{0.18pc}\xymatrix{
&&\cred{32} \ar@{-}[rd] \ar@{-}@[red][ld]\\
&\cred{31}\ar@{-}[rd] \ar@{-}@[red][ld]&& {\cgrn{\ncirc{22}}} \ar@{-}[rd] \ar@{-}[ld]\\
\cred{30}\ar@{-}@[red][rd] &&\cgrn{\ncirc{21}} \ar@{-}[rd] \ar@{-}[ld] & & \cgrn{\ncirc{12}} \ar@{-}[ld] \ar@{-}[rd]\\
&\cred{20}  \ar@{-}@[red][rd] && \cgrn{\ncirc{11}} \ar@{-}[rd] \ar@{-}[ld] & & 02 \ar@{-}[ld]\\ 
&& \cred{10}  \ar@{-}[rd] && 01 \ar@{-}[ld]\\
&&& 00
}
\qquad 
\xymatrixrowsep{0.4pc}\xymatrixcolsep{0.18pc}\xymatrix{
&&\cred{32} \ar@{-}[rd] \ar@{-}@[red][ld]\\
&\cred{31}\ar@{-}@[red][rd] \ar@{-}[ld]&& \cgrn{\ncirc{22}} \ar@{-}[rd] \ar@{-}[ld]\\
\cgrn{\ncirc{30}}\ar@{-}[rd] &&\cred{21} \ar@{-}[rd] \ar@{-}@[red][ld] & & \cgrn{\ncirc{12}} \ar@{-}[ld] \ar@{-}[rd]\\
&\cred{20}  \ar@{-}@[red][rd] && \cgrn{\ncirc{11}} \ar@{-}[rd] \ar@{-}[ld] & & 02 \ar@{-}[ld]\\ 
&& \cred{10}  \ar@{-}[rd] && 01 \ar@{-}[ld]\\
&&& 00
}
\qquad 
\xymatrixrowsep{0.4pc}\xymatrixcolsep{0.18pc}\xymatrix{
&&\cred{32} \ar@{-}@[red][rd] \ar@{-}[ld]\\
&\cgrn{\ncirc{31}}\ar@{-}[rd] \ar@{-}[ld]&& \cred{22} \ar@{-}[rd] \ar@{-}@[red][ld]\\
\cgrn{\ncirc{30}}\ar@{-}[rd] &&\cred{21} \ar@{-}[rd] \ar@{-}@[red][ld] & & \cgrn{\ncirc{12}} \ar@{-}[ld] \ar@{-}[rd]\\
&\cred{20}  \ar@{-}@[red][rd] && \cgrn{\ncirc{11}} \ar@{-}[rd] \ar@{-}[ld] & & 02 \ar@{-}[ld]\\ 
&& \cred{10}  \ar@{-}[rd] && 01 \ar@{-}[ld]\\
&&& 00
}
\]
\[
\xymatrixrowsep{0.4pc}\xymatrixcolsep{0.18pc}\xymatrix{
&&\cred{32} \ar@{-}[rd] \ar@{-}@[red][ld]\\
&\cred{31}\ar@{-}@[red][rd] \ar@{-}[ld]&& \cgrn{\ncirc{22}} \ar@{-}[rd] \ar@{-}[ld]\\
\cgrn{\ncirc{30}}\ar@{-}[rd] &&\cred{21} \ar@{-}@[red][rd] \ar@{-}[ld] & & \cgrn{\ncirc{12}} \ar@{-}[ld] \ar@{-}[rd]\\
&\cgrn{\ncirc{20}}  \ar@{-}[rd] && \cred{11} \ar@{-}[rd] \ar@{-}@[red][ld] & & 02 \ar@{-}[ld]\\ 
&& \cred{10}  \ar@{-}[rd] && 01 \ar@{-}[ld]\\
&&& 00
}
\qquad 
\xymatrixrowsep{0.4pc}\xymatrixcolsep{0.18pc}\xymatrix{
&&\cred{32} \ar@{-}@[red][rd] \ar@{-}[ld]\\
&\cgrn{\ncirc{31}}\ar@{-}[rd] \ar@{-}[ld]&& \cred{22} \ar@{-}[rd] \ar@{-}@[red][ld]\\
\cgrn{\ncirc{30}}\ar@{-}[rd] &&\cred{21} \ar@{-}@[red][rd] \ar@{-}[ld] & & \cgrn{\ncirc{12}} \ar@{-}[ld] \ar@{-}[rd]\\
&\cgrn{\ncirc{20}}  \ar@{-}[rd] && \cred{11} \ar@{-}[rd] \ar@{-}@[red][ld] & & 02 \ar@{-}[ld]\\ 
&& \cred{10}  \ar@{-}[rd] && 01 \ar@{-}[ld]\\
&&& 00
}
\qquad 
\xymatrixrowsep{0.4pc}\xymatrixcolsep{0.18pc}\xymatrix{
&&\cred{32} \ar@{-}@[red][rd] \ar@{-}[ld]\\
&\cgrn{\ncirc{31}}\ar@{-}[rd] \ar@{-}[ld]&& \cred{22} \ar@{-}@[red][rd] \ar@{-}[ld]\\
\cgrn{\ncirc{30}}\ar@{-}[rd] &&\cgrn{\ncirc{21}} \ar@{-}[rd] \ar@{-}[ld] & & \cred{12} \ar@{-}@[red][ld] \ar@{-}[rd]\\
&\cgrn{\ncirc{20}}  \ar@{-}[rd] && \cred{11} \ar@{-}[rd] \ar@{-}@[red][ld] & & 02 \ar@{-}[ld]\\ 
&& \cred{10}  \ar@{-}[rd] && 01 \ar@{-}[ld]\\
&&& 00
}
\]
\end{figure}

\begin{figure}[h]
\caption{The three pairs of lattice paths (shown in \cred{red} and \cblu{blue}) used to compute $\varphi_{2}(1,0)$ in $[0,3]\times
[0,2]$.  $A$-variable subscripts are underlined in \cgrn{\ncirc {green}}.}
\label{fig:3LP}
\[
\xymatrixrowsep{0.4pc}\xymatrixcolsep{0.18pc}\xymatrix{
&&32 \ar@{-}[rd] \ar@{-}[ld]\\
&\cred{31}\ar@{-}[rd] \ar@{-}@[red][ld]&& \cblu{22} \ar@{-}[rd] \ar@{-}@[blue][ld]\\
\cred{30}\ar@{-}@[red][rd] &&\cblu{21} \ar@{-}@[blue][rd] \ar@{-}[ld] & & \cgrn{\ncirc{12}} \ar@{-}[ld] \ar@{-}[rd]\\
&\cred{20}  \ar@{-}[rd] && \cblu{11} \ar@{-}[rd] \ar@{-}[ld] & & 02 \ar@{-}[ld]\\ 
&& 10  \ar@{-}[rd] && 01 \ar@{-}[ld]\\
&&& 00
}
\qquad 
\xymatrixrowsep{0.4pc}\xymatrixcolsep{0.18pc}\xymatrix{
&&32 \ar@{-}[rd] \ar@{-}[ld]\\
&\cred{31}\ar@{-}[rd] \ar@{-}@[red][ld]&& \cblu{22} \ar@{-}@[blue][rd] \ar@{-}[ld]\\
\cred{30}\ar@{-}@[red][rd] &&\cgrn{\ncirc{21}} \ar@{-}[rd] \ar@{-}[ld] & & \cblu{12} \ar@{-}@[blue][ld] \ar@{-}[rd]\\
&\cred{20}  \ar@{-}[rd] && \cblu{11} \ar@{-}[rd] \ar@{-}[ld] & & 02 \ar@{-}[ld]\\ 
&& 10  \ar@{-}[rd] && 01 \ar@{-}[ld]\\
&&& 00
}
\qquad 
\xymatrixrowsep{0.4pc}\xymatrixcolsep{0.18pc}\xymatrix{
&&32 \ar@{-}[rd] \ar@{-}[ld]\\
&\cred{31}\ar@{-}@[red][rd] \ar@{-}[ld]&& \cblu{22} \ar@{-}@[blue][rd] \ar@{-}[ld]\\
\cgrn{\ncirc{30}}\ar@{-}[rd] &&\cred{21} \ar@{-}[rd] \ar@{-}@[red][ld] & & \cblu{12} \ar@{-}@[blue][ld] \ar@{-}[rd]\\
&\cred{20}  \ar@{-}[rd] && \cblu{11} \ar@{-}[rd] \ar@{-}[ld] & & 02 \ar@{-}[ld]\\ 
&& 10  \ar@{-}[rd] && 01 \ar@{-}[ld]\\
&&& 00
}
\qquad 
\]
\end{figure}
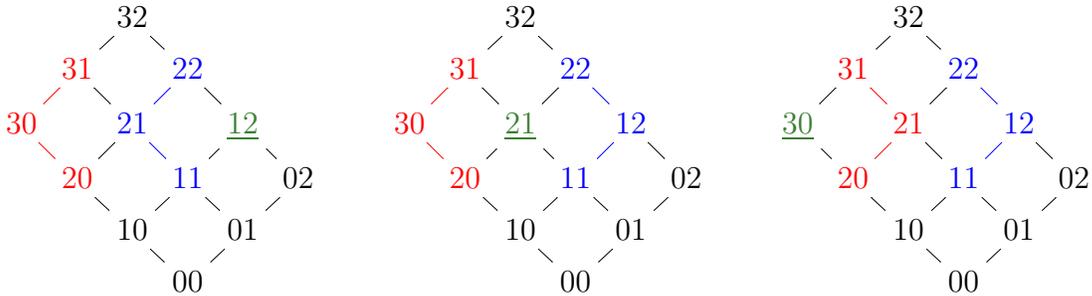


\begin{eg}\label{eg:main}
We use our main theorem to compute $\rho_{B}^{k+1}(2,1)$ for
$P = [0,3]\times [0,2]$ (the mirror image of the poset in Example~\ref{eg:Hasse}) for every $k\in \NN$.  
Here $r=3, s=2, i=2$, and $j=1$ throughout.   

\textbf{\bull When} $\mathbf{k=0}$, $M=0$ and we get 
$\ds \rho_{B}^{1}(2,1) = \frac{\varphi_{0}(2,1)}{\varphi_{1}(2,1)} =
\frac{A_{21}A_{22}A_{31}A_{32}}{A_{22}+A_{31}}.$  In general, we have
\begin{equation}\label{eq:phi0}
\varphi_{0}(i,j)=\prod_{(m,n)\geq (i,j)}A_{m,n}, 
\mathrm{~~where~the~product~runs~over~the~order~filter~of~}(i,j)\mathrm{~in~}P.
\end{equation}

\noindent (In this situation, there are 
no lattice paths to remove factors from the product.)  

\textbf{\bull When} $\mathbf{k=1}$, we still have $M = 0$, and $\ds \rho_{B}^{2}(2,1) = \frac{\varphi_{1}(1,0)}{\varphi_{2}(1,0)} = $
{\footnotesize \[
\frac{A_{11}A_{12}A_{21}A_{22} + A_{11}A_{12}A_{22}A_{30} + A_{11}A_{12}A_{30}A_{31} + 
A_{12}A_{20}A_{22}A_{30} + A_{12}A_{20}A_{30}A_{31} + A_{20}A_{21}A_{30}A_{31}}
	{A_{12}+A_{21} + A_{30}}.
\]}
For the numerator, $s_{1}= (1,0)$, $t_{1}=(3,2)$, and there are six lattice paths from $s_{1}$ to $t_{1}$,
each of which covers 5 elements and leaves 4 uncovered (Figure~\ref{fig:6LP}).  For the
denominator, $s_{1} = (2,0)$, $s_{2}= (1,1)$, $t_{1}=(3,1)$, and $t_{2}=(2,2)$, and each pair
of lattice paths leaves exactly one element uncovered (Figure~\ref{fig:3LP}). 

\textbf{\bull When} $\mathbf{k=2}$, we get $M=[2-2]_{+} + [2-1]_{+} = 1 \leq 2=k$.  So by part (a)
of the main theorem we have 
\[
\ds \rho_{B}^{3}(2,1) = \mu^{(1,0)}\left[ \frac{\varphi_{1}(1,0)}{\varphi_{2}(1,0)} \right] = 
    \text{(just shifting indices in the $k=1$ formula)}
\]
{\footnotesize \[
\frac{A_{01}A_{02}A_{11}A_{12} + A_{01}A_{02}A_{12}A_{20} + A_{01}A_{02}A_{20}A_{21} + 
A_{02}A_{10}A_{12}A_{20} + A_{02}A_{10}A_{20}A_{21} + A_{10}A_{11}A_{20}A_{21}}
	{A_{02}+A_{11} + A_{20}}.
\]}

\textbf{\bull When} $\mathbf{k=3}$, we get $M=[3-2]_{+} + [3-1]_{+} = 3 =k$. Therefore, 
\[
\rho_{B}^{4}(2,1) = \mu^{(2,1)}\left[ \frac{\varphi_{0}(2,1)}{\varphi_{1}(2,1)} \right] =
\mu^{(2,1)}\left[ \frac{A_{21}A_{22}A_{31}A_{32}}{A_{22}+A_{31}} \right] 
 = \frac{A_{00}A_{01}A_{10}A_{11}}{A_{01}+A_{10}}.
\] 
In this situation, we can also use part (b) of the main theorem to get 
\[
\rho_{B}^{4}(2,1) = 1/\rho_{B}^{3-2-1}(3-2, 2-1) = 1/\rho_{B}^{0}(1,1) = \frac{1}{x_{11}}. 
\]
The equality between these two expressions is easily checked as $$
\frac{A_{00}A_{01}A_{10}A_{11}}{A_{01}+A_{10}} = \frac{\frac{1}{x_{00}} 
\frac{x_{00}}{x_{01}} \frac{x_{00}}{x_{10}} \frac{x_{01} + x_{10}}{x_{11}}
}{\frac{x_{00}}{x_{01}} + \frac{x_{00}}{x_{10}}} =  \frac{1}{x_{11}}.$$  

\textbf{\bull When} $\mathbf{k=4}$, we get $M=[4-2]_{+} + [4-1]_{+} = 5 > k$. Therefore, 
by part (b) of the main theorem, then part (a),
\[
\rho_{B}^{5}(2,1) = 1/\rho_{B}^{4-2-1}(3-2, 2-1) = 1/\rho_{B}^{1}(1,1)
= \frac{\varphi_{1}(1,1)} {\varphi_{0}(1,1)} = \frac{A_{12}A_{22}+A_{12}A_{31}+A_{21}A_{31}}{A_{11}A_{12}A_{21}A_{22}A_{31}A_{32}}. 
\]
Each term in the numerator is associated with one of the three lattice paths from $(1,1)$ to
$(3,2)$ in $P$, while the denominator is just the product of all $A$-variables in the
principal order filter $\bigvee{(1,1)}$. 

\textbf{\bull When} $\mathbf{k=5}$, we get $M=[5-2]_{+} + [5-1]_{+} = 7 > k$. Therefore, 
by part (b) of the main theorem, then part (a),
\[
\rho_{B}^{6}(2,1) = 1/\rho_{B}^{5-2-1}(3-2, 2-1) = 1/\rho_{B}^{2}(1,1)
 = \frac{\varphi_{2}(0,0)} {\varphi_{1}(0,0)}\]
\[ = (A_{02}A_{12}+A_{02}A_{21}+A_{11}A_{21}+A_{30}A_{02}+A_{30}A_{11}+A_{30}A_{20}) \bigg / \]
 {\tiny \[(A_{01}A_{11}A_{02}A_{21}A_{12}A_{22} + 
 A_{01}A_{11}A_{02}A_{30}A_{12}A_{22} + 
 A_{01}A_{11}A_{02}A_{30}A_{12}A_{31} + 
 A_{01}A_{20}A_{02}A_{30}A_{12}A_{22} + 
 A_{01}A_{20}A_{02}A_{30}A_{12}A_{31} +\]
\[ A_{01}A_{20}A_{02}A_{30}A_{21}A_{31} + 
  A_{10}A_{20}A_{02}A_{30}A_{12}A_{31} + 
  A_{10}A_{20}A_{02}A_{30}A_{21}A_{31} + 
  A_{10}A_{20}A_{02}A_{30}A_{12}A_{22} + 
  A_{10}A_{20}A_{11}A_{30}A_{21}A_{31}).
\]}

\vspace{-1.5em}The numerator here represents the six pairs of NILPs from $s_{1}=(1,0)$ and $s_2 = (0,1)$ to $t_1 = (3,1)$ and $t_2 = (2,2)$. 
Each of the ten terms in the denominator corresponds to the complement of a lattice path from 
$s_1=(0,0)$ to $t_1 = (3,2)$.

\textbf{\bull When} $\mathbf{k=6}$, we get $M=[6-2]_{+} + [6-1]_{+} = 9 > k$. Therefore, 
by part (b) of the main theorem, then part (a),
\[
\rho_{B}^{7}(2,1) = 1/\rho_{B}^{6-2-1}(3-2, 2-1) = 1/\rho_{B}^{3}(1,1)
=  \mu^{(1,1)} \left[ \frac{ \varphi_1(1,1)}{\varphi_0(1,1)} \right] 
\]
\[
= \mu^{(1,1)} \left[ \frac{A_{12}A_{22} + A_{12}A_{31} + A_{21}A_{31}}
{A_{11}A_{11}A_{21}A_{22}A_{31}A_{32}}  \right] 
= \frac{A_{01}A_{11} + A_{01}A_{20} + A_{10}A_{20}}
{A_{00}A_{01}A_{10}A_{11}A_{20}A_{21}} = x_{21}.
\]
We will later see that this last equality is an application of Claim \ref{claim:Mk}, but one can also
deduce this by plugging in $A_{00} = 1/x_{00}, A_{10} = x_{00}/x_{10}, A_{01} = x_{00}/x_{01}, A_{11} = (x_{10}+x_{01})/x_{11}, A_{20} = x_{10}/x_{20},$ and 
$A_{21} = (x_{20}+x_{11})/x_{21}$.  Notice that periodicity also kicks in for this case and $\rho_B^7(2,1) = \rho_B^0(2,1) = x_{21}$ using 
$(r+s+2) = 7$.

\textbf{\bull When} $\mathbf{k\geq 6}$, we get by periodicity that $\rho_{B}^{k+1}(i,j)=\rho_{B}^{g}(i,j)$, where 
$g=k+1 \bmod 7$ has already been computed above.  

\end{eg}

\subsection{Applications of the path formula}\label{ss:files}
Our path formula has several applications, allowing us to give direct proofs of
interesting properties of birational rowmowtion on products of two chains, namely those displayed in
Examples~\ref{eg:2x2}--\ref{eg:2x2files}.  Our first two results were the original two main theorems
of~\cite{GrRo15}.

\begin{corollary}[{\cite[Thm.~30]{GrRo15}}]\label{cor:periodicity}
The birational rowmotion map $\rho_{B}$ on the product of two chains $P=[0,r]\times [0,s]$
is periodic, with period $r+s+2$.  
\end{corollary}
\begin{proof}
Apply part (b) of the main theorem twice, first with $k=r+s+1$, then with $k=(r-i)+(s-j)$
(checking in each case that $M\geq k$) to obtain 
\[
\rho_{B}^{r+s+2}(i,j) = 1/\rho_{B}^{r+s+1-i-j}(r-i, s-j) = 1\bigg/\frac{1}{\rho_{B}^{0}(i,j)} =
\rho_{B}^{0}(i,j). 
\]
\end{proof}

\begin{corollary}[{\cite[Thm.~32]{GrRo15}}]\label{cor:reciprocity}
The birational rowmotion map $\rho_{B}$ on the product of two chains $P=[0,r]\times [0,s]$
satisfies the following reciprocity: 
\[
\rho_{B}^{i+j+1}(i,j) = 1/\rho_{B}^{0}(r-i,s-j) = \frac{1}{x_{r-i,s-j}}.
\]
\end{corollary}
\begin{proof}
This is the special case $k=i+j$ in Theorem~\ref{thm:main} (b). 
\end{proof}

Our formula also allows us to give the first proof of a ``file homomesy'' for
birational rowmotion on the product of two chains stated by
Einstein and Propp~\cite[Thm.~9 and remarks below Cor.~7]{EiPr13}. For completeness, we
summarize the necessary background here.

\begin{definition}[{\cite[Def.~1]{PR13}}]\label{def:homomesy}
\global \def\fstat{g}
Given a set $\calS$, 
an invertible map $\tau$ from $\calS$ to itself such that each
$\tau$-orbit is finite, 
and a function (or ``statistic'') $\fstat: \calS \rightarrow K$
taking values in some field $K$ of characteristic zero,
we say 
the triple $(\calS,\tau,\fstat)$ 
exhibits {\bf homomesy}\footnote{Greek for ``same middle''} 
if there exists a constant $c \in K$
such that for every $\tau$-orbit $\calO \subset \calS$
\begin{equation}
\label{general-ce}
\frac{1} {\#\calO}
\sum_{x \in \calO} \fstat(x) = c .
\end{equation}
In this situation
we say that the function $\fstat : \calS \rightarrow K$ is 
{\bf homomesic}
under the action of $\tau$ on $\calS$,
or more specifically \textbf{\textit{c}-mesic}.
\end{definition}
When $\calS$ is a finite set,
homomesy can be restated equivalently as 
all orbit-averages being equal to the global average:
\begin{equation}
\label{eq:global}
\frac{1}{\#\calO} \sum_{x \in \calO} \fstat(x) =
\frac{1}{\#\calS} \sum_{x \in \calS} \fstat(x).
\end{equation}
One important example is that for the action of combinatorial rowmotion $\rho$ acting on
the set of order ideals $J(P)$, where $P = [0,r]\times [0,s]$, the cardinality statistic
$g=\#I$ is
$\frac{(r+1)(s+1)}{2}$-mesic.  But there are other homomesies for this action on $P$ as well,
e.g., for any fixed \emph{file} (see Example~\ref{eg:Hasse} and the preceding paragraph) $F$ of
$P$, the statistic $g=\#(I\cap F)$, which only counts the 
number of elements of $I$ within $F$ is homomesic.  
It is fruitful to consider these statistics as being
the sums of the indicator function statistics $\{\II_{x}:x\in P \}$, where for $I\in J(P)$, 
$\II_x(I) = 1$ if $x\in I$ and 0 otherwise. 
This is because linear combinations of such homomesic
statistics are also homomesic.  

As explained in \cite[\S 4.1]{PR13}, the collection of homomesic statistics that can be written as linear
combinations of the indicator statistics $\{\II_{x}:x\in P \}$  can all be
generated by just two kinds of fundamental homomesies: (a) $\II_{x}+\II_{y}$ where $x$ and $y$
are antipodal elements of the poset and (b) $\ds \sum_{x\in F}\II_{x}$, where $F$ is a file
of $P$.

For the detropicalized (or birational) version of homomesy on the rectangular poset
$P=[0,r]\times [0,s]$, the sums that define homomesy are transformed into products and the indicator statistics
$\II_{(i,j)}$ (for $(i,j) \in P$) are replaced 
by the statistic $\IIb_{(i,j)}(f): = f(i,j)$, i.e., simply the value of the birational
labeling $f$ at $(i,j)\in P$. 
Consequently the first
kind of fundamental homomesy becomes a ``geometric homomesy'' that 
(a) 
$\IIb_{(i,j)} \cdot
\IIb_{(r-i,s-j)} $ 
gives 1 when multiplied across a period of $\rho_{B}$
while the second kind is the same statement for
(b)
$\prod_{(i,j)\in F} \IIb_{(i,j)}$. 
The previous corollary (Corollary \ref{cor:reciprocity}) implies the first fundamental birational homomesy.  
The second fundamental birational homomesy is equivalent to the following
Theorem~\ref{thm:homomesy}, yielding the complete set of such birational homomesic statistics expected
for $\rho_{B}$. 
\begin{definition}\label{def:bi}
Given an action $\tau$ of period $n$ on a set of objects $S$ and a statistic $\xi :S\rightarrow \KK$, where
$\KK$ is any field, we call $\xi $ \textbf{birationally homomesic} 
if the value of $\prod_{k=0}^{n-1}\xi (\tau^{k}(s))$ is a constant $c\in \KK $,
\emph{independent} of $s$.  
\end{definition}

\begin{theorem} \label{thm:homomesy}
Given a choice of file
$F$ in $P = [0,r] \times [0,s]$, we have the identity  
$$\ds \prod_{k=0}^{r+s+1}\prod_{(i,j)\in F} \rho_{B}^{k}(i,j) = 1,$$ 
i.e., the statistic
$\prod_{(i,j)\in F} \IIb_{(i,j)}$ is birationally homomesic under
the action of birational rowmotion $\rho_B$.

More specifically, a choice of file $F$ is determined by the choice of an element on
the upper boundary, which may have the form $(r,d)$ for $0 \leq d \leq s$
or the form $(d,s)$ for $0 \leq d \leq r$.  Assuming without loss of
generality that $s \leq r$, this second case breaks further into two
subcases depending on whether $s \leq d$ or $d < s$.  Hence, the identity
above becomes one of the following double-product identities: 
\[
 \prod_{k=0}^{r+s+1} \prod_{c=0}^{d} \rho_B^{k+1}(r-c,d-c) = 1
 \mathrm{~~~if~~~} d < s \leq r, \qquad (a)
\]
\[
 \prod_{k=0}^{r+s+1} \prod_{c=0}^{d} \rho_B^{k+1}(d-c,s-c) = 1
 \mathrm{~~~if~~~} d < s \leq r, \qquad (b)
\]
\[
\prod_{k=0}^{r+s+1} \prod_{c=0}^{s} \rho_B^{k+1}(d-c,s-c) = 1 \mathrm{~~~if~~~} s \leq d \leq r. \qquad (c)
\]

\end{theorem}
Figure~\ref{fig:decomp} shows the decomposition of an example poset into the above cases.
We defer the proof of Theorem~\ref{thm:homomesy} to the next section, after the proof of
Theorem~\ref{thm:main}. 

\begin{figure}
\caption{The decomposition of $P=[0,4]\times [0,3]$ according to the three cases in
Theorem~\ref{thm:homomesy}: \cred{(a) with top element $(r,d)$ for $d<s$ (in red)},\\  
\cblu{(b) with top $(d,s)$ for $d < s$ (blue)},  or 
(c) with top $(d,s)$ for $d \geq  s$ (black).} 
\label{fig:decomp}
\hspace{2em}
\xymatrixrowsep{0.4pc}\xymatrixcolsep{0.18pc}\xymatrix{
&&& 43  \ar@{-}[ld]  \ar@{-}[rd]\\ 
&& \cred{42}  \ar@{-}@[red][ld] \ar@{-}[rd] && 33 \ar@{-}[ld] \ar@{-}[rd]\\
&\cred{41} \ar@{-}@[red][ld] \ar@{-}@[red][rd]&&32 \ar@{-}[rd] \ar@{-}[ld] && \cblu{23}  \ar@{-}[ld] \ar@{-}@[blue][rd]\\
\cred{40} \ar@{-}@[red][rd]&&\cred{31}\ar@{-}[rd] \ar@{-}@[red][ld]&& {22} \ar@{-}[rd] \ar@{-}[ld]&&\cblu{13}  \ar@{-}@[blue][ld] \ar@{-}@[blue][rd]\\
&\cred{30}\ar@{-}@[red][rd] &&21 \ar@{-}[rd] \ar@{-}[ld] & & \cblu{12} \ar@{-}[ld] \ar@{-}@[blue][rd]&&\cblu{03} \ar@{-}@[blue][ld]\\
&&\cred{20}  \ar@{-}[rd] && 11 \ar@{-}[rd] \ar@{-}[ld] & & \cblu{02} \ar@{-}@[blue][ld]\\ 
&&& 10  \ar@{-}[rd] && \cblu{01} \ar@{-}[ld]\\
&&&& 00
}
\end{figure}
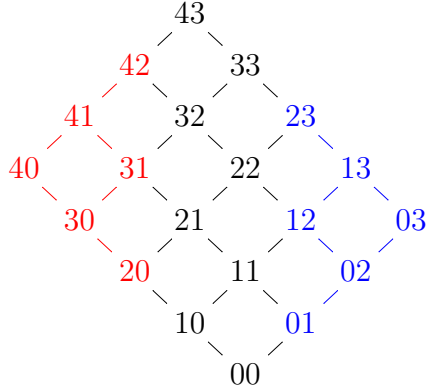

\section{Proof of Main Theorems}\label{sec:proofs}

\subsection{Special case $\mathbf{k=0}$}

We first prove Theorem~\ref{thm:main} for the special case when $k=0$, 
then for larger $k$ by induction working from the top of the
poset. We will often need to distinguish those elements on the \textbf{upper boundary} of
$P$, namely $\{(i,j): i=r  \text{ or } j=s \}$, each of which is covered by exactly
one element in $\whP$. All other elements of $P$ are covered by exactly two elements.  

As an initial case, at the top element $(\bfr ,\bfs )$ of $P$ we obtain 
\[
\rho_B^1(\bfr ,\bfs ) = \frac{x_{\bfr ,\bfs -1} + x_{\bfr -1,
\bfs }}{x_{\bfr , \bfs }} = A_{\bfr \bfs }  
\]
where the first equality is
by the definition of birational rowmotion (a single toggle), and the second is by the definition
of $A_{ij}$ (Equation~\eqref{eq:Aij}). 

Second, for any element $(\bfr, j)$ 
with $1\leq j\leq s$,
we assume by induction that $\rho_B^1(\bfr ,j) = \prod_{c=j}^{\bfs }
A_{\bfr ,c}$.  Then  
\[
\rho_B^1(\bfr ,j-1) = \frac{(x_{\bfr ,j-2}+x_{\bfr -1,j-1})(\rho_B^1(\bfr ,j))}{x_{\bfr ,j-1}} = \left(\frac{x_{\bfr ,j-2}+x_{\bfr -1,j-1}}{x_{\bfr ,j-1}}\right)\left(\prod_{c=j}^{\bfs }A_{\bfr ,c}\right) = \prod_{c=j-1}^{\bfs } A_{\bfr ,c},
\]
by the definition of $A_{\bfr ,j-1}$ and using the definition of birational rowmotion in the case where only a single
element covers it.   By symmetry we get a
formula for all (upper boundary) elements covered by a single element: 
\begin{equation}\label{eq:k=1;1cover}
\rho_B^1(\bfr ,j) = \prod_{c=j}^{\bfs } A_{\bfr ,c}, \text{ for } j\in [0,s] 
\qquad \text{ and } \qquad 
\rho_B^1(i,\bfs ) = \prod_{c=i}^{\bfr } A_{c, \bfs }, \text{ for } i\in [0,r]. 
\end{equation}
Note that this agrees with Theorem~\ref{thm:main} as follows.  Here $k=0$ forces $M=0$
and the shift transformation $\mu^{(0,0)}$ acts trivially.  Thus, Equation~\eqref{eq:main1} reduces to 
\[
\rho_{B}^{1}(i,j) = \frac{\varphi_{0}(i,j)}{\varphi_{1}(i,j)}. 
\]
The numerator is simply $\prod_{(u,v)\geq (i,j)} A_{uv}$, since the defining summation of
$\varphi_{0}(i,j)$ has only a single term, namely the empty tuple of lattice paths, and each
element of $\lpreg_{(i,j)}^{k} = \bigvee_{(i,j)}$ contributes one factor to that term's product.  
Hence, the value of the numerator agrees with the right-hand sides in Equation~\eqref{eq:k=1;1cover}.  
The denominator, $\varphi_{1}(i,j)$ for $i=\bfr $ or $j=\bfs $, equals $1$ because
there is a unique lattice path from $(i,j)$ to $(\bfr , \bfs) $ covering the entire
order filter $\bigvee_{(i,j)}$; thus, the summation consists of a single term, which is the
empty product, i.e., 1. 

All other elements $(i,j)$ of the poset are covered by exactly two elements.  In the case
that $(i,j)$ also covers (exactly) two elements, we obtain
\[
\rho_B^1(i,j) = \left(\frac{x_{i,j-1} + x_{i-1,j}}{x_{ij}}\right) \cdot \left( \rho_B^1(i,j+1) \parallel \rho_B^1(i+1,j)\right) 
= A_{ij} \cdot \left( \rho_B^1(i,j+1) \parallel \rho_B^1(i+1,j)\right).
\]

Recall that $a \parallel b$ denotes the parallel sum $\frac{1}{\frac{1}{a} + \frac{1}{b}}$.

\begin{rem} In the case that $(i,j)$ covers a single element of $P$, i.e., $i=0$ or $j=0$, recall that we defined $A_{i0}$, $A_{0j}$, and $A_{00}$ accordingly (see right after Equation~\eqref{eq:Aij}).  Thus,
\begin{equation}\label{eqn:2covers}
\rho_B^1(i,j) = A_{ij} \cdot \left( \rho_B^1(i,j+1) \parallel \rho_B^1(i+1,j)\right)
\end{equation}
holds for all $(i,j)$ covered by two elements (regardless of how many elements $(i,j)$ covers). 
\end{rem}

We claim that by induction (as long as $I+J > i+j$) that 
\[
\rho_B^1(I,J) =
\frac{1}{\sum_{\mathrm{~paths}~L} \frac{1}{\prod_{(p,q)\in P} A_{pq}}}
 = \sumpar_{\text{paths}~L} \prod_{(p,q)\in L} A_{pq},
\] where the sum is
over paths $L$ from the point $(I,J)$ up to the point $(\bfr ,\bfs )$.  (Here the large symbol $\sumpar$
denotes \emph{parallel summation}, the analogue of $\sum$ for parallel sums, which is
well-defined since $\parallel$ is associative and commutative.)  

In particular, in the special case of $\rho_B^1(\bfr ,j)$, there is a unique such path
and its weight is $\prod_{c=j}^{\bfs } A_{\bfr ,c}$, agreeing with the computation
above since the parallel sum of a single value is simply that value ($1/\frac{1}{a} = a$).  (By symmetry, we obtain 
$\rho_B^1(i,\bfs ) = \prod_{c=i}^{\bfr } A_{c,\bfs }$ as well.)  Then inductively, 

\begin{align*}
\rho_B^1(i,j) &= A_{ij}  \cdot \left( \sumpar_{\mathrm{~paths}~L~\text{from}~(i,j+1)}\prod_{(p,q)\in L} A_{pq} \ \ \  \parallel \sumpar_{\mathrm{~paths}~L~\text{from}~(i+1,j)}\prod_{(p,q)\in L} A_{pq} \right)\\ 
&= \sumpar_{\mathrm{~paths}~L~\text{from}~(i,j)}\prod_{(p,q)\in L} A_{pq},
\end{align*}
simply because every path from $(i,j)$  to $(\bfr , \bfs )$ must go either through
$(i,j+1)$ or $(i+1,j)$, and the $A_{ij}$ term distributes through.  
We finish the $k=0$ case by remarking that 
\begin{equation} \label{eq:k0} \rho_B^1(i,j) = \sumpar_{\mathrm{~paths}~L~\text{from}~(i,j)}\prod_{(p,q)\in L} A_{pq} = \frac{1}{\sum_{\mathrm{~paths}~L~\text{from}~(i,j)} \frac{1}{\prod_{(p,q)\in L} A_{pq}}}\end{equation}
$$ \hspace{-3em} = \frac{\varphi_0(i,j)}{\sum_{\mathrm{~paths}~L~\text{from}~(i,j)} \prod_{(p,q)\not\in L} A_{pq}} = \frac{\varphi_0(i,j)}{\varphi_1(i,j)},$$
where the second line comes from multiplying top and bottom by $\varphi_0(i,j)
=\prod_{(u,v)\geq (i,j)} A_{uv}$.  This
agrees with part (a) of our main theorem, where $k=0$ implies $[k-i]_{+} = [k-j]_{+} = M = 0$.  

\subsection{General case $\mathbf{k \geq 1}$}

We continue our proof by induction, starting by proving the case of $k=1$ on the upper boundary.  Then for each such $k$, we move downward through the entire rectangular poset by induction and then start again with a proof for the case of $(k+1)$ for the upper boundary.  To accomplish this proof we first verify two recurrence relations (Lemmas \ref{lem:bounce} and \ref{lem:genBounce}) that will be used for the induction step.  Both of these results are proven via combinatorial bijections.  Even though Lemma \ref{lem:bounce} looks like a special case, this result will imply Lemma \ref{lem:genBounce} and then Theorem \ref{thm:main} by verifying the recurrence used in our induction.

\begin{lemma} \label{lem:bounce} For $1 \leq k\leq \min \{ i,j \}$ we have the Pl\"ucker-like relation
\begin{align} \label{bounce-identity}
\begin{split}
\varphi_k(i-k,j-k) \varphi_{k-1}(i-k+1,j-k+1) &= \varphi_k(i-k+1,j-k) \varphi_{k-1}(i-k,j-k+1)\\ 
&+ \varphi_k(i-k,j-k+1) \varphi_{k-1}(i-k+1,j-k).
\end{split}
\end{align}
\end{lemma}

Since this statement involves pairs of families of \textbf{non-intersecting lattice paths},
abbreviated below to \textbf{NILPs}, we prove it via a colorful combinatorial bijection.

\begin{proof}
The definition of $\varphi_{k}$ (Equation~\eqref{eq:phi}) involves summing 
monomials in the $A_{ij}$'s, with each term corresponding to the elements left uncovered by a
$k$-tuple of NILPs.  So a term on the left-hand side of the Lemma
is represented by a pair of NILPs $(\cblu{\calB}, \cred{\calR})$ 
offset from one another by one rank.  Example~\ref{eg:bouncePaths} gives an example to
illustrate both this and the bijection below.  
Specifically, the lower NILPs $\cblu{\calB}$, whose endpoints are marked with $\cblu{\circ}$,
represents a monomial from $\varphi_{k}(i-k,j-k)$, 
and the upper NILPs $\cred{\calR}$, whose endpoints are marked by $\cred{\times}$, represents one
from $\varphi_{k-1}(i-k+1,j-k+1)$.  Our goal is to transform this pair into a pair of NILPs
counted by one of the terms on the right-hand side of Lemma~\ref{lem:bounce}.  

Starting from the bottom $\cblu{\circ}$'s (lowest points in $\cblu{\calB}$), we create two \textbf{bounce
paths} and $(k-2)$ \textbf{twigs} as follows.  From the leftmost $\cblu{\circ}$ on the bottom, move up blue edges, i.e., edges in $\cblu{\calB}$, until encountering a vertex with a downward red
edge,  i.e., an edge in $\cred{\calR}$.  Then move down red edges until encountering a vertex with an upward blue edge.
Continue in this way, reversing directions whenever possible and only traversing unused
edges, until a terminal vertex is reached.  (No such path can terminate at an internal
vertex, since any edge by which one enters must be paired with a possible exit.)  Do the same procedure starting from the 
rightmost $\cblu{\circ}$ on the bottom.  We refer to both of these paths as \textbf{bounce paths}.

Since we reverse directions along bounce paths in a systematic way, we always follow blue edges upward and red edges downward.  
In addition to these two bounce paths, the $(k-2)$ $\cblu{\circ}$'s in the interior of the
bottom immediately connect to a $\cred{\times}$ in the rank second from the bottom.  We
refer to these blue edges as {\bf twigs}.  Since the twigs cover all but $1$ of the $(k-2)$
$\cred{\times}$'s, only one of the two bounce paths may return to the bottom of the poset
ending with a segment of downward red edges.   
Furthermore the $(k-1)$ red paths, starting from the $\cred{\times}$'s at the top, intersect $(k-1)$ of the $k$ topmost $\cblu{\circ}$'s, leaving only  $1$  $\cblu{\circ}$ untouched.  Note that a bounce path can only end at the top of the poset if it does not meet a red path that it can follow downward.
It follows that one of these two bounce paths ends at the top of the interval, at the
$\cblu{\circ}$ on top untouched by the red paths, and the other bounce path ends at the
bottom of the interval, at the $\cred{\times}$ on the bottom not covered by a twig.  We call the former
a \textbf{vertical bounce path} and the latter a \textbf{horizontal bounce path}.  (Note that in the case that all blue paths point northeast (resp.\ northwest) starting from the bottom   $\cblu{\circ}$'s, the horizontal bounce path turns out to be a twig as well.  The remainder of our procedure is consistent whether or not we treat this as a twig or as a horizontal bounce path.)  

We proceed by interchanging the colors of the edges along the horizontal bounce path, along all the twigs, and swap the $\cred{\times}$ and $\cblu{\circ}$ endpoints at the bottom, while leaving the remaining edges of $\cblu{\calB} \cup \cred{\calR}$ unchanged (also leaving the colors of the vertical bounce path unchanged).  
We then truncate the vertical bounce path by deleting the bottommost edge.  These transformations result in a new pair of lattice path families which we denote as $(\cblu{\calB}', \cred{\calR}')$.
The bottom endpoints for $\cblu{\calB}'$ will be one step either to the northeast or northwest of the original ones, indicating
respectively whether it is contributing to the first or second summand on the right-hand side of
Lemma~\ref{lem:bounce}.  The bottom endpoints of $\cred{\calR}'$ are skewed in the other
direction, i.e., the southwest or the southeast, respectively.   
 
Furthermore, if the lattice paths $L_B \in \cblu{\calB}$ and $L_R \in \cred{\calR}$ did not
originally intersect, then their edges would not lie along any bounce path.  Consequently,
$L_B$ would be a lattice path again in $\cblu{\calB}'$ unchanged\footnote{With the small
exception of possibly truncating the bottommost leftmost or rightmost edge.  However, even
this change would not affect intersections.}, and the same is true for $L_R$ in
$\cred{\calR}'$.  They would again not intersect.  On the other hand, if $L_B$ and $L_R$ did
originally intersect, then they could meet along a bounce path.  Being part of larger NILPs,
$L_B$ would not intersect any path in $\cblu{\calB}$ and $L_R$ would not intersect any path
in $\cred{\calR}$. Swapping colors of individual edges along $L_B$ and $L_R$ might break
this intersection-free property, but since all colors of edges along a horizontal bounce
path are swapped simultaneously, we ensure that each collection of paths, $\cblu{\calB}'$
and $\cred{\calR}'$, is still intersection-free.

Hence, the result is a new pair of NILPs $(\cblu{\calB}', \cred{\calR}')$ with the lower
endpoints of $\cblu{\calB}'$ on the second rank from the bottom of the interval skewed left
(resp.~right) while the lower endpoints of $\cred{\calR}'$ are on the bottom rank of the
interval and skewed right (resp.~left).  By construction, this map is well-defined, and
$\cblu{\calB}'$ is a collection of $k$ lattice paths from $\cblu{\circ}$'s to
$\cblu{\circ}$'s, and  $\cred{\calR}'$ is a collection of $(k-1)$ lattice paths from
$\cred{\times}$'s to $\cred{\times}$'s. 
Thus the new pair represents a pair of monomials counted by
$\varphi_k(i-k,j-k+1) \varphi_{k-1}(i-k+1,j-k)$ in the former case, and counted by
$\varphi_k(i-k+1,j-k) \varphi_{k-1}(i-k,j-k+1)$ in the latter case.

Finally this procedure is reversible, yielding the desired bijection.  In particular, given
a pair of NILPs $(\cblu{\mathcal{B}}'€™, \cred{\mathcal{R}}')$, which has the lower endpoints
of $\cblu{\mathcal{B}}'$ skewed left (resp. right) while the lower endpoints of
$\cred{\mathcal{R}}'$ are skewed right (resp. left), we build a vertical bounce path
starting from the leftmost (resp. rightmost) lower endpoint of $\cblu{\mathcal{B}}'$ and a
horizontal bounce path starting from the rightmost (resp. leftmost) lower endpoint of
$\cred{\mathcal{R}}'$.  Swapping colors along the horizontal bounce path and the twigs
(defined similarly to as above) yields a centrally symmetric pair of NILPs
$(\cblu{\mathcal{B}},\cred{\mathcal{R}})$.  The validity of this construction follows by the
same argument which we used above.

\end{proof}

\begin{eg}\label{eg:bouncePaths}
Let $k=5$, and consider the following pair of families of NILPs, $(\cblu{\calB}, \cred{\calR})$ 
shown in (\cblu{blue}, \cred{red}) in $\lpreg_{(i-5,j-5)}^{5}\cup
\lpreg_{(i-4,j-4)}^{4}$, with $r-i=s-j=2$. (Double edges shown in {\color{Plum}Plum} are used 
to represent one edge of each color because of limitations of our drawing package.)  
\[
\begin{array}{c}
\xymatrixrowsep{0.9pc}\xymatrixcolsep{0.20pc}\xymatrix{
& & & & \cred{\times} \ar@{--}@[red][ld] & & \cred{\times} \ar@{--}@[red][rd] & & \cred{\times} \ar@{--}@[red][rd] & & \cred{\times} \ar@{--}@[red][rd] & & & & \\
& & & \cblu{\circ} \ar@{-}@[blue][ld] \ar@{--}@[red][rd] & & \cblu{\circ} \ar@{-}@[blue][ld] & & \cblu{\circ} \ar@{-}@[blue][ld] \ar@{--}@[red][rd] & & \cblu{\circ} \ar@{=}@[Plum][rd] & & \cblu{\circ} \ar@{=}@[Plum][rd] & & & \\
& & \bullet \ar@{-}@[blue][ld]  & & \bullet \ar@{--}@[red][ld]\ar@{-}@[blue][rd] & & \bullet \ar@{-}@[blue][rd] & & \bullet \ar@{--}@[red][ld] & & \bullet \ar@{--}@[red][ld] \ar@{-}@[blue][rd] & & \bullet  \ar@{--}@[red][ld] \ar@{-}@[blue][rd] & & \\
& \bullet \ar@{-}@[blue][ld] & & \bullet \ar@{--}@[red][rd] & & \bullet \ar@{-}@[blue][ld] & & \bullet \ar@{--}@[red][ld] \ar@{-}@[blue][rd]& & \bullet \ar@{--}@[red][rd]& & \bullet \ar@{-}@[blue][ld] \ar@{--}@[red][rd] & & \bullet \ar@{-}@[blue][rd] & \\
\bullet \ar@{-}@[blue][rd] & & \bullet  & & \bullet \ar@{--}@[red][ld] \ar@{-}@[blue][rd]& & \bullet \ar@{--}@[red][ld] & & \bullet \ar@{-}@[blue][rd] & & \bullet \ar@{--}@[red][ld]\ar@{-}@[blue][rd]& & \bullet \ar@{--}@[red][ld]& & \bullet \ar@{-}@[blue][ld] \\
& \bullet \ar@{-}@[blue][rd] & & \bullet \ar@{--}@[red][rd] & & \bullet \ar@{=}@[Plum][rd] & & \bullet & & \bullet \ar@{=}@[Plum][ld] & & \bullet \ar@{=}@[Plum][ld] & & \bullet \ar@{-}@[blue][ld] & \\
& & \bullet \ar@{-}@[blue][rd] & & \cred{\times} & & \cred{\times} \ar@{-}@[blue][ld] & & \cred{\times} \ar@{-}@[blue][ld] & & \cred{\times} \ar@{-}@[blue][ld] & & \bullet \ar@{-}@[blue][ld] & & \\ 
& & & \cblu{\circ} & & \cblu{\circ} & & \cblu{\circ} & & \cblu{\circ} & & \cblu{\circ} & & &
}
\end{array}
\]

We create bounce paths and twigs as follows.  
\[
\begin{array}{c}
\xymatrixrowsep{0.9pc}\xymatrixcolsep{0.20pc}\xymatrix{
& & & & \cred{\times} & & \cred{\times} & & \cred{\times} & & \cred{\times} & & & & \\
& & & \cblu{\circ} \ar@{<-}@[blue][ld] \ar@{-->}@[red][rd]& & \cblu{\circ} \ar@{<-}@[blue][ld] & & \cblu{\circ} & & \cblu{\circ} & & \cblu{\circ} & & & \\
& & \bullet \ar@{<-}@[blue][ld]  & & \bullet & & \bullet & & \bullet & & \bullet \ar@{-->}@[red][ld] \ar@{<-}@[blue][rd] & & \bullet  \ar@{-->}@[red][ld] \ar@{<-}@[blue][rd] & & \\
& \bullet \ar@{<-}@[blue][ld]  & & \bullet & & \bullet & & \bullet \ar@{-->}@[red][ld] \ar@{<-}@[blue][rd]& & \bullet \ar@{-->}@[red][rd]& & \bullet \ar@{<-}@[blue][ld] \ar@{-->}@[red][rd] & & \bullet \ar@{<-}@[blue][rd] & \\
\bullet \ar@{<-}@[blue][rd] & & \bullet  & & \bullet \ar@{-->}@[red][ld] \ar@{<-}@[blue][rd]& & \bullet \ar@{-->}@[red][ld] & & \bullet \ar@{<-}@[blue][rd] & & \bullet \ar@{-->}@[red][ld]\ar@{<-}@[blue][rd]& & \bullet \ar@{-->}@[red][ld] & & \bullet \ar@{<-}@[blue][ld] \\
& \bullet \ar@{<-}@[blue][rd] & & \bullet \ar@{-->}@[red][rd] & & \bullet & & \bullet & & \bullet & & \bullet & & \bullet \ar@{<-}@[blue][ld] & \\
& & \bullet \ar@{<-}@[blue][rd] & & \cred{\times} & & \cred{\times} \ar@{<-}@[blue][ld]  & & \cred{\times} \ar@{<-}@[blue][ld] & & \cred{\times} \ar@{<-}@[blue][ld] & & \bullet \ar@{<-}@[blue][ld] & & \\ 
& & & \cblu{\circ} & & \cblu{\circ} & & \cblu{\circ} & & \cblu{\circ} & & \cblu{\circ} & & &
}
\end{array}
\]

Note that the leftmost bounce path is vertical, i.e., it ends at the top, so its colors remain the same.  The rightmost (horizontal) bounce path traverses the poset as follows: \cblu{NE}, \cblu{NE}, \cblu{NE}, \cblu{NW}, \cblu{NW}, \cred{SW}, \cblu{NW}, \cred{SW}, \cred{SE}, \cblu{NE}, \cred{SE}, \cred{SW}, \cblu{NW}, \cred{SW}, \cblu{NW}, \cblu{NW}, \cred{SW}, \cred{SW}, \cblu{NW}, \cred{SW}, \cred{SE}.  We interchange the colors along the twigs and the rightmost bounce path, which is horizontal.
\[
\begin{array}{c}
\xymatrixrowsep{0.9pc}\xymatrixcolsep{0.20pc}\xymatrix{
& & & & \cred{\times} & & \cred{\times} & & \cred{\times} & & \cred{\times} & & & & \\
& & & \cblu{\circ} & & \cblu{\circ} & & \cblu{\circ} & & \cblu{\circ} & & \cblu{\circ} & & & \\
& & \bullet  & & \bullet & & \bullet & & \bullet & & \bullet \ar@{--}@[blue][ld] \ar@{-}@[red][rd] & & \bullet  \ar@{--}@[blue][ld] \ar@{-}@[red][rd] & & \\
& \bullet & & \bullet & & \bullet & & \bullet \ar@{--}@[blue][ld] \ar@{-}@[red][rd]& & \bullet \ar@{--}@[blue][rd]& & \bullet \ar@{-}@[red][ld] \ar@{--}@[blue][rd] & & \bullet \ar@{-}@[red][rd] & \\
\bullet & & \bullet  & & \bullet \ar@{--}@[blue][ld] \ar@{-}@[red][rd]& & \bullet \ar@{--}@[blue][ld] & & \bullet \ar@{-}@[red][rd] & & \bullet \ar@{--}@[blue][ld]\ar@{-}@[red][rd]& & \bullet \ar@{--}@[blue][ld] & & \bullet \ar@{-}@[red][ld] \\
& \bullet & & \bullet \ar@{--}@[blue][rd] & & \bullet & & \bullet & & \bullet & & \bullet & & \bullet \ar@{-}@[red][ld] & \\
& & \bullet & & \cred{\times}  & & \cred{\times} \ar@{-}@[red][ld]  & & \cred{\times} \ar@{-}@[red][ld] & & \cred{\times} \ar@{-}@[red][ld] & & \bullet \ar@{-}@[red][ld] & & \\ 
& & & \cblu{\circ} & & \cblu{\circ} & & \cblu{\circ} & & \cblu{\circ} & & \cblu{\circ} & & &
}
\end{array}
\]
We then fill in the original edges (with their original colors) and swap $\cred{\times}$ and
$\cblu{\circ}$ at the bottom.
\[
\begin{array}{c}
\xymatrixrowsep{0.9pc}\xymatrixcolsep{0.20pc}\xymatrix{
& & & & \cred{\times} \ar@{--}@[red][ld]& & \cred{\times} \ar@{--}@[red][rd] & & \cred{\times} \ar@{--}@[red][rd] & & \cred{\times} \ar@{--}@[red][rd] & & & & \\
& & & \cblu{\circ} \ar@{-}@[blue][ld] \ar@{--}@[red][rd] & & \cblu{\circ} \ar@{-}@[blue][ld] & & \cblu{\circ} \ar@{-}@[blue][ld] \ar@{--}@[red][rd] & & \cblu{\circ} \ar@{=}@[Plum][rd] & & \cblu{\circ} \ar@{=}@[Plum][rd] & & & \\
& & \bullet \ar@{-}@[blue][ld]  & & \bullet \ar@{--}@[red][ld]\ar@{-}@[blue][rd] & & \bullet \ar@{-}@[blue][rd] & & \bullet \ar@{--}@[red][ld] & & \bullet \ar@{-}@[blue][ld] \ar@{--}@[red][rd] & & \bullet  \ar@{-}@[blue][ld] \ar@{--}@[red][rd] & & \\
& \bullet \ar@{-}@[blue][ld] & & \bullet \ar@{--}@[red][rd] & & \bullet \ar@{-}@[blue][ld] & & \bullet \ar@{-}@[blue][ld] \ar@{--}@[red][rd]& & \bullet \ar@{-}@[blue][rd]& & \bullet \ar@{--}@[red][ld] \ar@{-}@[blue][rd] & & \bullet \ar@{--}@[red][rd] & \\
\bullet \ar@{-}@[blue][rd] & & \bullet  & & \bullet \ar@{-}@[blue][ld] \ar@{--}@[red][rd]& & \bullet \ar@{-}@[blue][ld] & & \bullet \ar@{--}@[red][rd] & & \bullet \ar@{-}@[blue][ld]\ar@{--}@[red][rd]& & \bullet \ar@{-}@[blue][ld]& & \bullet \ar@{--}@[red][ld] \\
& \bullet \ar@{-}@[blue][rd] & & \bullet \ar@{-}@[blue][rd] & & \bullet \ar@{=}@[Plum][rd] & & \bullet & & \bullet \ar@{=}@[Plum][ld] & & \bullet \ar@{=}@[Plum][ld] & & \bullet \ar@{--}@[red][ld] & \\
& & \bullet \ar@{-}@[blue][rd]& & \cblu{\circ} & & \cblu{\circ} \ar@{--}@[red][ld] & & \cblu{\circ} \ar@{--}@[red][ld] & & \cblu{\circ} \ar@{--}@[red][ld] & & \bullet \ar@{--}@[red][ld] & & \\
& & & \cred{\times}  & & \cred{\times} & & \cred{\times} & & \cred{\times} & & \cred{\times} & & & \\
}
\end{array}
\]
Finally, we shorten the vertical bounce path by one edge, replacing $\cred{\times} \mapsto
\bullet$ with $\bullet \mapsto \cblu{\circ}$ so that the new starting point of the blue path
is at the same level as the other paths in 
$\cblu{\calB'}$.  

\[
\begin{array}{c}
\xymatrixrowsep{0.9pc}\xymatrixcolsep{0.20pc}\xymatrix{
& & & & \cred{\times} \ar@{--}@[red][ld]& & \cred{\times} \ar@{--}@[red][rd] & & \cred{\times} \ar@{--}@[red][rd] & & \cred{\times} \ar@{--}@[red][rd] & & & & \\
& & & \cblu{\circ} \ar@{-}@[blue][ld] \ar@{--}@[red][rd] & & \cblu{\circ} \ar@{-}@[blue][ld] & & \cblu{\circ} \ar@{-}@[blue][ld] \ar@{--}@[red][rd] & & \cblu{\circ} \ar@{=}@[Plum][rd] & & \cblu{\circ} \ar@{=}@[Plum][rd] & & & \\
& & \bullet \ar@{-}@[blue][ld]  & & \bullet \ar@{--}@[red][ld]\ar@{-}@[blue][rd] & & \bullet \ar@{-}@[blue][rd] & & \bullet \ar@{--}@[red][ld] & & \bullet \ar@{-}@[blue][ld] \ar@{--}@[red][rd] & & \bullet  \ar@{-}@[blue][ld] \ar@{--}@[red][rd] & & \\
& \bullet \ar@{-}@[blue][ld] & & \bullet \ar@{--}@[red][rd] & & \bullet \ar@{-}@[blue][ld] & & \bullet \ar@{-}@[blue][ld] \ar@{--}@[red][rd]& & \bullet \ar@{-}@[blue][rd]& & \bullet \ar@{--}@[red][ld] \ar@{-}@[blue][rd] & & \bullet \ar@{--}@[red][rd] & \\
\bullet \ar@{-}@[blue][rd] & & \bullet  & & \bullet \ar@{-}@[blue][ld] \ar@{--}@[red][rd]& & \bullet \ar@{-}@[blue][ld] & & \bullet \ar@{--}@[red][rd] & & \bullet \ar@{-}@[blue][ld]\ar@{--}@[red][rd]& & \bullet \ar@{-}@[blue][ld]& & \bullet \ar@{--}@[red][ld] \\
& \bullet \ar@{-}@[blue][rd] & & \bullet \ar@{-}@[blue][rd] & & \bullet \ar@{=}@[Plum][rd] & & \bullet & & \bullet \ar@{=}@[Plum][ld] & & \bullet \ar@{=}@[Plum][ld] & & \bullet \ar@{--}@[red][ld] & \\
& & \cblu{\circ} & & \cblu{\circ} & & \cblu{\circ} \ar@{--}@[red][ld] & & \cblu{\circ} \ar@{--}@[red][ld] & & \cblu{\circ} \ar@{--}@[red][ld] & & \bullet \ar@{--}@[red][ld] & & \\
& & & \bullet & & \cred{\times} & & \cred{\times} & & \cred{\times} & & \cred{\times} & & & \\
}
\end{array}
\]

The result is a new pair of NILPs $(\cblu{\calB}', \cred{\calR}')$.  In this example, the
lower endpoints of $\cblu{\calB}'$ are now skewed left, representing a monomial in
$\varphi_{5}(i-4, j-5)$, while those of $\cred{\calR}'$ are skewed right, 
representing a monomial in $\varphi_{4}(i-5, j-4)$.  In other examples, the skewing will
be opposite, giving a pair $(\cblu{\calB}', \cred{\calR}')$ corresponding to a pair of
monomials counted by $\varphi_{5}(i-5, j-4)\varphi_{4}(i-4,j-5)$.  

\end{eg}

The next lemma allows us to handle cases where shifting the point $(i,j)$ by $(-k,-k)$ lands
outside of the poset $P$.  In such cases we shift the point back inside $P$ so
that the lattice paths are well defined, shifting the indices of
the $A$-variables accordingly.      
\begin{lemma}\label{lem:genBounce}
For $i,j,k$ such that $(i,j) \in [0,r]\times [0,s]$ and 
$\left([k-i]_+ + [k-j]_+\right) \leq k \leq (r+s+1)$, we have the Pl\"ucker-like relation:
\begin{small}
$$\mu^{([k-j]_+,[k-i]_+)} \varphi_{k-M_{00}}(i-k+M_{00},j-k+M_{00})
~\mu^{([k-j-1]_+,[k-i-1]_+)} \varphi_{k-1-M_{11}}(i-k+1+M_{11},j-k+1+M_{11})$$
$$ = \mu^{([k-j]_+,[k-i-1]_+)} \varphi_{k-M_{10}}(i-k+1+M_{10},j-k+M_{10})  
~\mu^{([k-j-1]_+,[k-i]_+)} \varphi_{k-1-M_{01}}(i-k+M_{01},j-k+1+M_{01})$$
$$+ \mu^{([k-j-1]_+,[k-i]_+)} \varphi_{k-M_{01}}(i-k+M_{01},j-k+1+M_{01}) 
~\mu^{([k-j]_+, [k-i-1]_+)} \varphi_{k-1-M_{10}}(i-k+1+M_{10},j-k+M_{10}),$$
where $M_{00} = [k-i]_+ + [k-j]_+$, $M_{11} = [k-i-1]_+ + [k-j-1]_+$, $M_{01} = [k-i]_+ +
[k-j-1]_+$, and $M_{10} = [k-i-1]_+ + [k-j]_+$.  
\end{small}
\end{lemma}

This result includes Lemma \ref{lem:bounce} as a special case since  $1 \leq k\leq \min \{ i,j \}$ implies that $[k-i]_+ = [k-j]_+ = 0$ and therefore $\left([k-i]_+ + [k-j]_+\right) \leq k \leq (r+s+1)$ immediately holds.  We use Lemma \ref{lem:genBounce} to complete the proof of case (a) of Theorem \ref{thm:main}.  Part (b) is handled by a separate argument.

\begin{proof} 
We prove this more general case by extending the domain where Lemma \ref{lem:bounce} holds and
then specializing to the case we need.  In particular, we extend the rectangular poset
$[0,r]\times [0,s]$ by embedding it inside $\{-r-s,-r-s+1,\dots, r-1, r\} \times
\{-r-s,-r-s+1,\dots, s-1,s\}$.
Inside of this larger rectangular poset, build the order filter with base $(i-k, j-k)$
noting that each or both of these coordinates could now be negative (and therefore would
have been outside the original $[0,r]\times [0,s]$ poset). 

Let $\Phi_k(i-k,j-k)$ denote the set of non-intersecting lattice paths (NILPs) in this order filter.  Following Lemma \ref{lem:bounce}, we have a combinatorial bijection 
$$\Phi_k(i-k,j-k) \times \Phi_{k-1}(i-k+1,j-k+1) \to \Phi_k(i-k+1,j-k) \times \Phi_{k-1}(i-k,j-k+1)$$ 
$$\bigcup \Phi_k(i-k,j-k+1) \times \Phi_{k-1}(i-k+1,j-k)$$
where the right-hand side is a disjoint union.

We let $\Phi_k(a,b)^{(c,d)}$ be shorthand for the subset of NILPs in the order filter based at point $(a,b)$ such that the lattice paths from $\{s_1,s_2,\dots, s_c\}$ to $\{t_1,t_2,\dots, t_c\}$ (ordered left-to-right in their respective ranks of $\{-r-s,-r-s+1,\dots, r\} \times \{-r-s,-r-s+1,\dots,s\}$) each traverse the leftmost possible route in the order filter and the lattice paths from $\{s_{k-d+1}, s_{k-d+2},\dots, s_{k}\}$ to $\{t_{k-d+1}, t_{k-d+2},\dots, t_{k}\}$ traverse the rightmost routes.  We refer to such NILPs as {\bf (c,d)-boundary hugging}.  This notation is well-defined whenever $k \geq c+d$.

Using this notation, we claim that the above restricts to a bijection 
\begin{eqnarray} \label{eq:restrictbounce} \Phi_k(i-k,j-k)^{([k-j]_+,[k-i]_+)} \times \Phi_{k-1}(i-k+1,j-k+1)^{([k-j-1]_+,[k-i-1]_+)} \to 
\\ \nonumber \Phi_k(i-k+1,j-k)^{([k-j]_+,[k-i-1]_+)}\times \Phi_{k-1}(i-k,j-k+1)^{([k-j-1]_+,[k-i]_+)} \\ 
\nonumber \bigcup
\Phi_k(i-k,j-k+1)^{([k-j-1]_+,[k-i]_+)} \times \Phi_{k-1}(i-k+1,j-k)^{([k-j]_+,[k-i-1]_+)}\end{eqnarray}
Note that we have assumed that $k \geq [k-i]_+ + [k-j]_+$, so all six of these sets are well-defined.

To see this, consider Figure \ref{fig:boundary-hugging} (a).  
Consider a pair of NILPs, satisfying the boundary hugging restriction, associated to $ \Phi_k(i-k,j-k)^{([k-j]_+,[k-i]_+)} \times \Phi_{k-1}(i-k+1,j-k+1)^{([k-j-1]_+,[k-i-1]_+)}$.
Without loss of generality, assume in our configuration that the {\bf horizontal bounce path} starts from the {\bf rightmost} $\cblu{\circ}$.  We hence swap the colors on the right but leave them unchanged on the left except for the twigs.  (If the horizontal bounce path starts from the leftmost $\cblu{\circ}$ instead, we use the mirror image.) 

After this swap, we have a configuration that has the form of Figure \ref{fig:boundary-hugging} (b).  
In particular the pattern of upward steps and downward steps starting from the bottom of the horizontal bounce path changes the colors of the rightmost boundary hugging blue and red paths in a predictable way.  Furthermore, the horizontal bounce path cannot reach the leftmost $[k-j]_+$ boundary hugging blue paths since the red steps are downward steps and cannot intersect the blue steps that are pointed down and to the left.  
The resulting configuration after the swap corresponds to a pair of NILPs associated to the product
$\Phi_k(i-k+1,j-k)^{([k-j]_+,[k-i-1]_+)}\times \Phi_{k-1}(i-k,j-k+1)^{([k-j-1]_+,[k-i]_+)}$.  

Furthermore, because we have used boundary hugging paths as constructed above, the only elements of these order filters left uncovered by any of these six sets of NILPs are elements that are in the original $[0,r]\times [0,s]$ poset, i.e., with nonnegative entries for both coordinates.  Consequently, the map defined by (\ref{eq:restrictbounce}) yields a weight-preserving-bijection after weighting NILPs $\mathcal{L}$ by the products of the $A_{cd}$'s for points $(c,d) \in [0,r]\times [0,s]$ left uncovered by $\mathcal{L}$.  
We end up associating the lattice paths in $\Phi_k(i-k,j-k)^{([k-j]_+,[k-i]_+)}$ to an order filter that has the element $(\bfr - [k-j]_+,\bfs-[k-i]_+)$ as its top (rather than $(\bfr,\bfs)$).  

We then obtain Lemma \ref{lem:genBounce} as written by translating the bottom and top of the order filter.  Hence, for each $\epsilon_i,\epsilon_j \in \{0,1\}$, the subset
$\Phi_k(i-k+\epsilon_i,j-k+\epsilon_j)^{([k-j-\epsilon_j]_+,[k-i-\epsilon_i]_+)}$ has 
$$\mu^{([k-j-\epsilon_j]_+,[k-i-\epsilon_i]_+)} \varphi_{k-M_{\epsilon_i \epsilon_j}}(i-k+M_{\epsilon_i \epsilon_j},j-k+M_{\epsilon_i \epsilon_j})$$
as its generating function.
\end{proof}

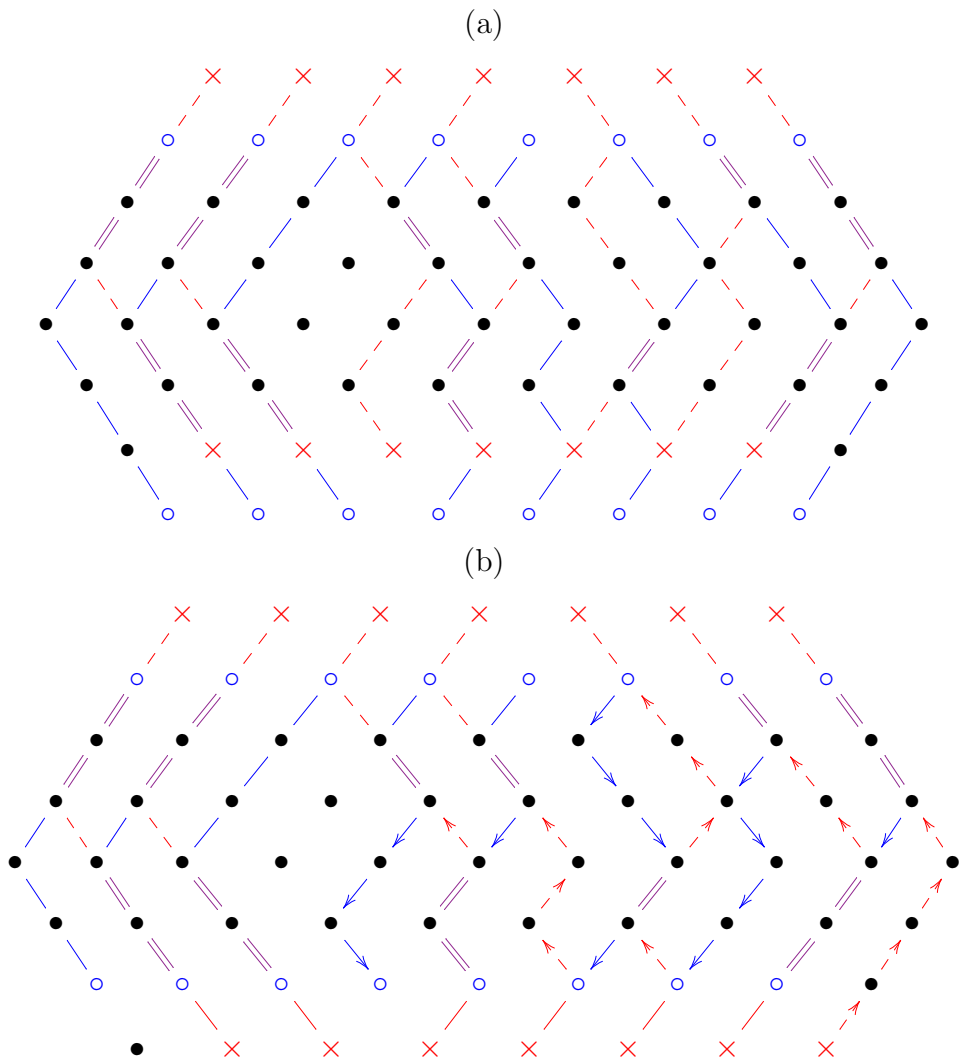
\begin{figure}[h]
(a) \[
\begin{array}{c}
\xymatrixrowsep{0.9pc}\xymatrixcolsep{0.20pc}\xymatrix{
& & & & \cred{\times} \ar@{--}@[red][ld] & & \cred{\times} \ar@{--}@[red][ld] & & \cred{\times} \ar@{--}@[red][ld] & & \cred{\times} \ar@{--}@[red][ld] & & \cred{\times} \ar@{--}@[red][rd] & & \cred{\times} \ar@{--}@[red][rd] & & \cred{\times} \ar@{--}@[red][rd] & & & & \\
& & & \cblu{\circ} \ar@{=}@[Plum][ld] & & \cblu{\circ} \ar@{=}@[Plum][ld] & & \cblu{\circ} \ar@{-}@[blue][ld] \ar@{--}@[red][rd] & & \cblu{\circ} \ar@{-}@[blue][ld] \ar@{--}@[red][rd] & & \cblu{\circ} \ar@{-}@[blue][ld] & & \cblu{\circ} \ar@{--}@[red][ld] \ar@{-}@[blue][rd] & & \cblu{\circ} \ar@{=}@[Plum][rd] & & \cblu{\circ} \ar@{=}@[Plum][rd] & & & \\
& & \bullet \ar@{=}@[Plum][ld]  & & \bullet \ar@{=}@[Plum][ld]  & & \bullet \ar@{-}@[blue][ld]  & & \bullet \ar@{=}@[Plum][rd]  & & \bullet \ar@{=}@[Plum][rd] & & \bullet \ar@{--}@[red][rd] & & \bullet \ar@{-}@[blue][rd] & & \bullet \ar@{--}@[red][ld] \ar@{-}@[blue][rd] & & \bullet \ar@{=}@[Plum][rd] & & \\
& \bullet \ar@{-}@[blue][ld]\ar@{--}@[red][rd]  & & \bullet \ar@{-}@[blue][ld]\ar@{--}@[red][rd]  & & \bullet \ar@{-}@[blue][ld] & & \bullet \ar@{-}@[white][ld] & & \bullet \ar@{--}@[red][ld] \ar@{-}@[blue][rd] & & \bullet \ar@{--}@[red][ld] \ar@{-}@[blue][rd] & & \bullet \ar@{--}@[red][rd] & & \bullet \ar@{-}@[blue][ld] \ar@{--}@[red][rd] & & \bullet \ar@{-}@[white][ld] \ar@{-}@[blue][rd] & & \bullet \ar@{--}@[red][ld] \ar@{-}@[blue][rd] & \\
\bullet \ar@{-}@[blue][rd] & & \bullet \ar@{=}@[Plum][rd] & &  \bullet \ar@{=}@[Plum][rd] & &  \bullet \ar@{-}@[white][rd] & & \bullet \ar@{--}@[red][ld]  & & \bullet \ar@{=}@[Plum][ld] & & \bullet \ar@{-}@[blue][ld] & & \bullet \ar@{=}@[Plum][ld] & & \bullet \ar@{--}@[red][ld]\ar@{-}@[white][rd]& & \bullet \ar@{=}@[Plum][ld]& & \bullet \ar@{-}@[blue][ld] \\
& \bullet \ar@{-}@[blue][rd] & & \bullet \ar@{=}@[Plum][rd] & & \bullet \ar@{=}@[Plum][rd] & & \bullet \ar@{--}@[red][rd] & & \bullet \ar@{=}@[Plum][rd] & & \bullet \ar@{-}@[blue][rd] & & \bullet \ar@{--}@[red][ld]\ar@{-}@[blue][rd] & & \bullet \ar@{--}@[red][ld] & & \bullet \ar@{=}@[Plum][ld] & & \bullet \ar@{-}@[blue][ld] & \\
& & \bullet \ar@{-}@[blue][rd] & & \cred{\times}\ar@{-}@[blue][rd] & & \cred{\times}\ar@{-}@[blue][rd] & & \cred{\times} \ar@{-}@[white][ld] & & \cred{\times} \ar@{-}@[blue][ld] & & \cred{\times} \ar@{-}@[blue][ld] & & \cred{\times} \ar@{-}@[blue][ld] & & \cred{\times} \ar@{-}@[blue][ld] & & \bullet \ar@{-}@[blue][ld] & & \\ 
& & & \cblu{\circ} & & \cblu{\circ} & & \cblu{\circ} & & \cblu{\circ} & & \cblu{\circ} & & \cblu{\circ} & & \cblu{\circ} & & \cblu{\circ} & & &
}
\end{array}
\]

(b) \[
\begin{array}{c}
\xymatrixrowsep{0.9pc}\xymatrixcolsep{0.20pc}\xymatrix{
& & & & \cred{\times} \ar@{--}@[red][ld] & & \cred{\times} \ar@{--}@[red][ld] & & \cred{\times} \ar@{--}@[red][ld] & & \cred{\times} \ar@{--}@[red][ld] & & \cred{\times} \ar@{--}@[red][rd] & & \cred{\times} \ar@{--}@[red][rd] & & \cred{\times} \ar@{--}@[red][rd] & & & & \\
& & & \cblu{\circ} \ar@{=}@[Plum][ld] & & \cblu{\circ} \ar@{=}@[Plum][ld] & & \cblu{\circ} \ar@{-}@[blue][ld] \ar@{--}@[red][rd] & & \cblu{\circ} \ar@{-}@[blue][ld] \ar@{--}@[red][rd] & & \cblu{\circ} \ar@{-}@[blue][ld] & & \cblu{\circ} \ar@{->}@[blue][ld] \ar@{<--}@[red][rd] & & \cblu{\circ} \ar@{=}@[Plum][rd] & & \cblu{\circ} \ar@{=}@[Plum][rd] & & & \\
& & \bullet \ar@{=}@[Plum][ld]  & & \bullet \ar@{=}@[Plum][ld]  & & \bullet \ar@{-}@[blue][ld]  & & \bullet \ar@{=}@[Plum][rd]  & & \bullet \ar@{=}@[Plum][rd] & & \bullet \ar@{->}@[blue][rd] & & \bullet \ar@{<--}@[red][rd] & & \bullet \ar@{->}@[blue][ld] \ar@{<--}@[red][rd] & & \bullet \ar@{=}@[Plum][rd] & & \\
& \bullet \ar@{-}@[blue][ld]\ar@{--}@[red][rd]  & & \bullet \ar@{-}@[blue][ld]\ar@{--}@[red][rd]  & & \bullet \ar@{-}@[blue][ld] & & \bullet \ar@{-}@[white][ld] & & \bullet \ar@{->}@[blue][ld] \ar@{<--}@[red][rd] & & \bullet \ar@{->}@[blue][ld] \ar@{<--}@[red][rd] & & \bullet \ar@{->}@[blue][rd] & & \bullet \ar@{<--}@[red][ld] \ar@{->}@[blue][rd] & & \bullet \ar@{-}@[white][ld] \ar@{<--}@[red][rd] & & \bullet \ar@{->}@[blue][ld] \ar@{<--}@[red][rd] & \\
\bullet \ar@{-}@[blue][rd] & & \bullet \ar@{=}@[Plum][rd] & &  \bullet \ar@{=}@[Plum][rd] & &  \bullet \ar@{-}@[white][rd] & & \bullet \ar@{->}@[blue][ld]  & & \bullet \ar@{=}@[Plum][ld] & & \bullet \ar@{<--}@[red][ld] & & \bullet \ar@{=}@[Plum][ld] & & \bullet \ar@{->}@[blue][ld]\ar@{-}@[white][rd]& & \bullet \ar@{=}@[Plum][ld]& & \bullet \ar@{<--}@[red][ld] \\
& \bullet \ar@{-}@[blue][rd] & & \bullet \ar@{=}@[Plum][rd] & & \bullet \ar@{=}@[Plum][rd] & & \bullet \ar@{->}@[blue][rd] & & \bullet \ar@{=}@[Plum][rd] & & \bullet \ar@{<--}@[red][rd] & & \bullet \ar@{->}@[blue][ld]\ar@{<--}@[red][rd] & & \bullet \ar@{->}@[blue][ld] & & \bullet \ar@{=}@[Plum][ld] & & \bullet \ar@{<--}@[red][ld] & \\
& & \cblu{\circ}  & &\cblu{\circ}\ar@{-}@[red][rd] & &\cblu{\circ}\ar@{-}@[red][rd] & &\cblu{\circ} \ar@{-}@[white][ld] & &\cblu{\circ} \ar@{-}@[red][ld] & &\cblu{\circ} \ar@{-}@[red][ld] & &\cblu{\circ} \ar@{-}@[red][ld] & &\cblu{\circ} \ar@{-}@[red][ld] & & \bullet \ar@{<--}@[red][ld] & & \\ 
& & & \bullet  & &  \cred{\times} & &  \cred{\times} & &  \cred{\times} & &  \cred{\times} & &  \cred{\times} & &  \cred{\times} & &  \cred{\times} & & &
}
\end{array}
\]

\caption{(a) Illustrating a pair of NILPs which are {\color{blue}$(3,2)$-boundary hugging} and {\color{red}$(2,1)$-boundary hugging}, respectively. (b) After applying our bijection, we have {\color{blue}$(3,1)$-boundary hugging NILPs} and {\color{red}$(2,2)$-boundary hugging NILPs}.}
\label{fig:boundary-hugging}
\end{figure}

\begin{eg}
We consider the NILPs illustrated in Figure \ref{fig:boundary-hugging} (a).  In this  case, $[k-j]_+ = 3$ and $[k-i]_+ = 2$ so that we have $(3,2)$-boundary hugging blue paths and $(2,1)$-boundary hugging red paths.  
In this example, the {\bf horizontal bounce path} starts from the {\bf rightmost} $\cblu{\circ}$ (as opposed to the leftmost $\cblu{\circ}$).
After swapping colors along the horizontal bounce path, we get the NILPs of Figure
\ref{fig:boundary-hugging} (b) with $(3,1)$-boundary hugging blue paths, which are
left-justified, and $(2,2)$-boundary hugging red paths, which are right-justified.  
\end{eg}

\vspace{-1em}

\subsection{Proof of Theorem \ref{thm:main} (a)}

We consider the off-boundary case where $(i,j)$ both covers and is covered by two elements
of $P$. 
  Under this hypothesis, we have the following identity by the definition of birational rowmotion:
$$\rho_B^{k+1}(i,j) = \frac{\bigg(\rho_B^k(i,j-1) + \rho_B^k(i-1,j)\bigg) \bigg( \rho_B^{k+1}(i+1,j) \parallel \rho_B^{k+1}(i,j+1) \bigg)}{\rho_B^k(i,j)}.$$

By induction on $k$, and the fact that we apply birational rowmotion from top to bottom, we can rewrite this formula as  
$$\frac{ \bigg(\frac{A}{B} + \frac{C}{D}\bigg) \bigg( 
\frac{B}{G} \parallel \frac{D}{H}\bigg)}{\frac{E}{F}} = 
\frac{ \bigg(\frac{A}{B} + \frac{C}{D}\bigg) \bigg( 
\frac{BD}{DG ~+~ BH}\bigg)}{\frac{E}{F}}$$
$$= \bigg(\frac{AD + BC}{BD}\bigg) \bigg( 
\frac{BD}{DG ~+~ BH}\bigg)\bigg(\frac{F}{E}\bigg) 
= \frac{D\frac{A}{E} + B\frac{C}{E}}{D\frac{G}{F} + B\frac{H}{F}} $$

where

$$A = \mu^{([k-j]_+,
[k-1-i]_+)}\varphi_{k-1-M_{10}}(i-k+1+M_{10},j-k+M_{10})$$

$$B = \mu^{([k-j]_+,
[k-1-i]_+)}\varphi_{k-M_{10}}(i-k+1+M_{10},j-k+M_{10})$$

$$C = \mu^{([k-1-j]_+,
[k-i]_+)}\varphi_{k-1-M_{01}}(i-k+M_{01},j-k+1+M_{01})$$

$$D = \mu^{([k-i]_+,
[k-1-j]_+)}\varphi_{k-M_{01}}(i-k+M_{01},j-k+1+M_{01})$$

$$E = \mu^{([k-1-j]_+,
[k-1-i]_+)}\varphi_{k-1-M_{11}}(i-k+1+M_{11},j-k+1+M_{11})$$

$$F = \mu^{([k-1-j]_+,
[k-1-i]_+)}\varphi_{k-M_{11}}(i-k+1+M_{11},j-k+1+M_{11})$$

$$G = \mu^{([k-j]_+,
[k-1-i]_+)}\varphi_{k+1-M_{10}}(i-k+1+M_{10},j-k+M_{10})$$

$$H = \mu^{([k-1-j]_+,
[k-i]_+)}\varphi_{k+1-M_{01}}(i-k+M_{01},j-k+1+M_{01})$$
using $M_{\epsilon_i, \epsilon_j} = [k-i-\epsilon_i]_+ + [k-j-\epsilon_j]_+$ for $\epsilon_i, \epsilon_j \in \{0,1\}$.

To prove Theorem \ref{thm:main} (a), it therefore suffices to prove, using this above shorthand, that
$$D\frac{A}{E} + B\frac{C}{E} = \mu^{([k-j]_+,
[k-i]_+)}\varphi_{k-M_{00}}(i-k+M_{00},j-k+M_{00})$$
and 
$$D\frac{G}{F} + B\frac{H}{F} = \mu^{([k-j]_+,
[k-i]_+)}\varphi_{k+1-M_{00}}(i-k+M_{00},j-k+M_{00}).$$

Letting $\alpha = \mu^{([k-j]_+,
[k-i]_+)}\varphi_{k-M_{00}}(i-k+M_{00},j-k+M_{00})$ and \\ $\beta = \mu^{([k-j]_+,
[k-i]_+)}\varphi_{k+1-M_{00}}(i-k+M_{00},j-k+M_{00})$, 
we note that these two equations, i.e., 
$$\alpha E = BC + DA \mathrm{~~and~~} \beta F = GD + HB,$$
both follow from two applications\footnote{
The second equation follows from substituting $k$ by $k+1$, $i$ by $i+1$, and $j$ by $j+1$.} of Lemma \ref{lem:genBounce}.

\begin{rem}
The proof is analogous in cases where the element $(i,j)$ covers (or is covered by) only a
single element, with some of the terms in the above expression being replaced with a $1$ or a $0$.  
\end{rem}

\subsection{Proof of Theorem \ref{thm:main} (b)}

Before continuing with the proof in the case when $M> k$, we note the following simplified
formula in the special case when $M=k$. 

\begin{claim} \label{claim:Mk}
Under the hypotheses of Theorem \ref{thm:main}, if $M=k$ (i.e., $i+j=k$)
then $$\rho_B^{k+1}(i,j) = \mu^{(i,j)} \left(\frac{\varphi_{0}(i,j)}{\varphi_{1}(i,j)}\right)
= \mu^{(i,j)} \rho_B^1(i,j) = \frac{1}{x_{r-i,s-j}}.$$ 
\end{claim}

\begin{proof} The first two equalities follow from Theorem \ref{thm:main} (a),
while we prove the last equality as follows.  Since the principal order filter
$\bigvee_{(i,j)}$ is isomorphic to the product of chains $[0,r-i]\times [0,r-j]$, we easily
reduce the claim to the case
$i=j=k=0$, i.e., it suffices to show the last equality of 
\[
\rho_B^1(0,0) = \frac{\varphi_{0}(0,0)}{\varphi_{1}(0,0)} =
\frac{\prod_{p=0}^r \prod_{q=0}^s A_{pq}}
{\sum_{ \calL  \in S_{1}(0,0)}
\hspace{2em}\prod_{\stackrel{(i,j) \in \lpreg_{(0,0)}^{1}}{(i,j) \not \in L_1}} A_{ij}. 
} = \frac{1}{x_{r,s}}. 
\] 
In this situation, our family of lattice paths reduces to a single lattice path $L_{1}$,
the numerator can be thought of as $\ds \prod_{(p,q)\in P}A_{pq}$, and $\lpreg_{(0,0)}^{1} =
P$ as well.  By clearing denominators and dividing through by the double-product we
equivalently need to show the following: 
\begin{claim} \label{claim:00}
\[
\sum_{ \calL  \in S_{1}(0,0)}
\kern 1em \prod_{{(i,j) \in L_1}} A_{ij}^{-1} = x_{r,s}. 
\]
\end{claim}
\textit{Proof.}
For the base case $s=0$, we get that $P$ is a chain of length $r$ and the only lattice path
consists of
every element of $P$.  In this case $A_{i0} = \frac{x_{i,0}}{x_{i-1,0}}$ for $i\in [r]$, with
$A_{00} = x_{00}$, so the single summand is the telescoping product 
\[
\frac{x_{0,0}}{1}\cdot \frac{x_{1,0}}{x_{0,0}}\cdot \frac{x_{2,0}}{x_{1,0}}\cdot 
\dotsb  \frac{x_{r,0}}{x_{r-1,0}} = x_{r,0}
\]
as required.  Symmetrically, the claim also holds for $r=0$ and any $s$.  Now suppose that
$rs>0$ and that
the claim holds for every rectangular poset whose dimensions are strictly smaller than
$[0,r]\times [0,s]$.  Set $\frakL(p,q): = \{\text{lattice paths from } (0,0)\text{ to }
(p,q)\}$.   Any lattice path
from $(0,0)$ to $(r,s)$ must go through either $(r-1,s)$
or $(r,s-1)$.  Thus,
\begin{align*}
\sum_{ \calL  \in S_{1}(0,0)} \kern 1em \prod_{{(i,j) \in L_1}} A_{ij}^{-1} 
& = A_{r,s}^{-1} \sum_{L\in \frakL (r-1,s)} \kern 1em \prod_{{(i,j) \in L}} A_{ij}^{-1} 
    +A_{r,s}^{-1} \sum_{L\in \frakL (r,s-1)} \kern 1em \prod_{{(i,j) \in L}} A_{ij}^{-1}\\ 
& = A_{r,s}^{-1} (x_{r-1,s} + x_{r,s-1})\\ 
&= x_{r,s}, 
\end{align*}
using the induction hypothesis and the definition of $A_{i,j}$.  This finishes the proofs of
both Claim~\ref{claim:00} and Claim~\ref{claim:Mk}.  
\end{proof}

We next consider the case when $M = k+1$ (i.e., $i+j = k-1$).  We start with the degenerate case
$\rho_B^2(0,0) =   \frac{\rho_B^{2}(1,0) \parallel \rho_B^{2}(0,1)}{\rho_B^1(0,0)}
= \frac{\frac{1}{x_{r-1,s}} \parallel \frac{1}{x_{r,s-1}} }{\frac{1}{x_{rs}}} = \frac{ \frac{1}{x_{r-1,s}+x_{r,s-1}}}{\frac{1}{x_{rs}}}=
 \frac{1}{A_{rs}}.$  Note here that we used Claim \ref{claim:Mk} to simplify the calculations.
Continuing by induction\footnote{As above, if $(i,j)$ only covers one element, we have a single summand rather than two inside the left parenthesis.},
$$\rho_B^{k+1}(i,j) = \frac{\bigg(\rho_B^k(i,j-1) + \rho_B^k(i-1,j)\bigg)  \bigg( \rho_B^{k+1}(i+1,j) \parallel \rho_B^{k+1}(i,j+1) \bigg)}{\rho_B^k(i,j)}.$$
Using Claim \ref{claim:Mk} and the inductive hypothesis, the right-hand side simplifies to 
$$\frac{\bigg(\frac{1}{\rho_B^{k-1-(i+j-1)}(r-i,s-j+1)} + \frac{1}{\rho_B^{k-1-(i+j-1)}(r-i+1,s-j)} \bigg) 
 \bigg( 1/x_{r-i-1,s-j} \parallel 1/x_{r-i,s-j-1} \bigg)}{1/x_{r-i,s-j}}
=$$ $$\bigg(\frac{1}{\rho_B^{1}(r-i,s-j+1)} + \frac{1}{\rho_B^{1}(r-i+1,s-j)} \bigg)
 \frac{1}{A_{r-i,s-j}}.$$
Using Equation (\ref{eq:k0}), we can expand this out further as (assuming $i+j=k-1$)
$$\rho_B^{k+1}(i,j) = \bigg(
\sum_{\mathrm{~paths}~L~\text{from}~(r-i,s-j+1)} \frac{1}{\prod_{(p,q)\in L} A_{pq}}
~~~ + \sum_{\mathrm{~paths}~L~\text{from}~(r-i+1,s-j)} \frac{1}{\prod_{(p,q)\in L} A_{pq}}\bigg) \frac{1}{A_{r-i,s-j}}.$$
Any lattice path connecting $(r-i,s-j)$ to $(r,s)$ either goes through $(r-i,s-j+1)$ or
through $(r-i+1,s-j)$.  Combining these into a single sum over lattice paths, we get 
$$\rho_B^{k+1}(i,j)  = \sum_{\mathrm{~paths}~L~\text{from}~(r-i,s-j)} \frac{1}{\prod_{(p,q)\in L} A_{pq}} = \frac{\varphi_1(r-i,s-j)}{\varphi_0(r-i,s-j)} =  \frac{1}{\rho_B^{1}(r-i,s-j)},$$ 
agreeing with Theorem \ref{thm:main} (b) when $k=i+j+1$.

Lastly, when $M > k+1$, we use Theorem \ref{thm:main} (b) inductively to obtain
$$\rho_B^{k+1}(i,j) = \frac{\bigg(\rho_B^k(i,j-1) + \rho_B^k(i-1,j)\bigg)  \bigg( \rho_B^{k+1}(i+1,j) \parallel \rho_B^{k+1}(i,j+1) \bigg)}{\rho_B^k(i,j)}$$
$$ = \frac{\bigg(\frac{1}{\rho_B^{k-i-j}(r-i,s-j+1)} + \frac{1}{\rho_B^{k-i-j}(r-i+1,s-j)}\bigg)  \bigg( \frac{1}{\rho_B^{k-1-i-j}(r-i-1,s-j)} \parallel \frac{1}{\rho_B^{k-1-i-j}(r-i,s-j-1)} \bigg)}
{\frac{1}{\rho_B^{k-1-i-j}(r-i,s-j)}}$$
$$ = \Bigg(\frac{\bigg(\rho_B^{K+1}(I,J+1) \parallel \rho_B^{K+1}(I+1,J)\bigg) \bigg(\rho_B^{K}(I-1,J) + \rho_B^{K}(I,J-1)\bigg)}{\rho_B^{K}(I,J)} \Bigg)^{-1} = \frac{1}{\rho_B^{K+1}(I,J)}$$
where $K = k-1-i-j$, $I = r-i$, $J=s-j$.  This finishes the proof.

\subsection{Proof of file homomesy}\label{ss:fileProof}

In this section we use our main theorem to prove the file-homomesy result,
Theorem~\ref{thm:homomesy}.  The proof is a mixture of straighforward cancellations directly
from our formula and some subtle recombinations of terms, leading to a double-counting
argument to show two products are equal.  We start with an illustrative example that shows
the initial cancellations.   

\begin{eg}\label{eg:fileTable}
Let $(r,s) = (4,3)$, and $d=2$, with corresponding file $F = \{(4,2), (3,1), (2,0) \}$.  The
following table displays the values (in terms of the $\varphi$-polynomials) taken on by each
element of the file across 
a $\rho_{B}$-period.

\[
\setlength\arraycolsep{1em}
\def\arraystretch{2.5}
\begin{array}{c|c c c}
&		  (4,2)&		  (3,1)&	(2,0)\\ 
\hline
k=0&\dfrac{\varphi_{0}(4,2)}{\varphi_{1}(4,2)} & \dfrac{\varphi_{0}(3,1)}{\cred{\varphi_{1}(3,1)}} & \dfrac{\varphi_{0}(2,0)}{\cblu{\varphi_{1}(2,0)}}\\ 
k=1&\dfrac{\cred{\varphi_{1}(3,1)}}{\varphi_{2}(3,1)} & \dfrac{\cblu{\varphi_{1}(2,0)}}{\cred{\varphi_{2}(2,0)}} & \mu^{(1,0)}\left[\dfrac{\varphi_{0}(2,0)}{\cblu{\varphi_{1}(2,0)}}\right]\\
k=2&\dfrac{\cred{\varphi_{2}(2,0)}}{\varphi_{3}(2,0)} & \mu^{(1,0)}\left[\dfrac{\cblu{\varphi_{1}(2,0)}}{\cred{\varphi_{2}(2,0)}}\right] & \mu^{(2,0)}\left[\dfrac{\varphi_{0}(2,0)}{\cblu{\varphi_{1}(2,0)}}\right] = \frac{1}{x_{23}}\\ 
k=3&\mu^{(1,0)}\left[\dfrac{\cred{\varphi_{2}(2,0)}}{\varphi_{3}(2,0)}\right] & \mu^{(2,0)}\left[\dfrac{\cblu{\varphi_{1}(2,0)}}{\cred{\varphi_{2}(2,0)}}\right]&\dfrac{\varphi_{1}(2,3)}{\varphi_{0}(2,3)}\\ 
k=4& \mu^{(2,0)}\left[\dfrac{\cred{\varphi_{2}(2,0)}}{\varphi_{3}(2,0)}\right]&\mu^{(3,1)}\left[\dfrac{\varphi_{0}(3,1)}{\cred{\varphi_{1}(3,1)}}\right] = \frac{1}{x_{12}}&\dfrac{\varphi_{2}(1,2)}{\cgrn{\varphi_{1}(1,2)}}\\
k=5& \mu^{(3,1)}\left[\dfrac{\cred{\varphi_{1}(3,1)}}{\varphi_{2}(3,1)}\right]&\dfrac{\cgrn{\varphi_{1}(1,2)}}{\varphi_{0}(1,2)}&\dfrac{\varphi_{3}(0,1)}{\cgrn{\varphi_{2}(0,1)}}\\ 
k=6& \mu^{(4,2)}\left[\dfrac{\varphi_{0}(4,2)}{\varphi_{1}(4,2)}\right] = \frac{1}{x_{01}}&\dfrac{\cgrn{\varphi_{2}(0,1)}}{\cblu{\varphi_{1}(0,1)}}&\mu^{(0,1)}\left[\dfrac{\varphi_{3}(0,1)}{\cgrn{\varphi_{2}(0,1)}}\right]\\ 
k=7& \dfrac{\cblu{\varphi_{1}(0,1)}}{\varphi_{0}(0,1)}&\mu^{(0,1)}\left[\dfrac{\cgrn{\varphi_{2}(0,1)}}{\cblu{\varphi_{1}(0,1)}}\right]&\mu^{(1,2)}\left[\dfrac{\varphi_{2}(1,2)}{\cgrn{\varphi_{1}(1,2)}}\right]\\ 
k=8& \mu^{(0,1)}\left[\dfrac{\cblu{\varphi_{1}(0,1)}}{\varphi_{0}(0,1)}\right] = x_{42}&\mu^{(1,2)}\left[\dfrac{\cgrn{\varphi_{1}(1,2)}}{\varphi_{0}(1,2)}\right] = x_{31}&\mu^{(2,3)}\left[\dfrac{\varphi_{1}(2,3)}{\varphi_{0}(2,3)}\right] = x_{20}\\ 
\end{array}
\]

We color code entries in \cred{red}, \cblu{blue}, and \cgrn{green} to pair numerators of one entry which agree with
denominators of another entry, hence cancelling in the product of all values.  
The remaining entries either are equal to $1$ or cancel each other out, as handled below.  

For convenience we record them here, listing them down columns from
left-to-right: $$\big[\varphi_0(4,2) \varphi_1(4,2)^{-1}  \varphi_2(3,1)^{-1} \varphi_3(2,0)^{-1}
\mu^{(1,0)} \varphi_3(2,0)^{-1}  \mu^{(2,0)} \varphi_3(2,0)^{-1} \mu^{(3,1)}  
\varphi_2(3,1)^{-1}  \mu^{(4,2)}\varphi_0(4,2)$$  $$\mu^{(4,2)} \varphi_1(4,2)^{-1}
\varphi_0(0,1)^{-1} \mu^{(0,1)} \varphi_0(0,1)^{-1} \big] \cdot 
\big[\varphi_0(3,1)  \mu^{(3,1)} \varphi_0(3,1) \varphi_0(1,2)^{-1} \mu^{(1,2)} \varphi_0(1,2)^{-1}\big]\cdot $$ 
$$\big[\varphi_0(2,0) \mu^{(1,0)} \varphi_0(2,0) 
\mu^{(2,0)} \varphi_0(2,0) \varphi_1(2,3) \varphi_0(2,3)^{-1} \varphi_2(1,2) \varphi_3(0,1) \mu^{(0,1)} \varphi_3(0,1)  \mu^{(1,2)} \varphi_2(1,2)$$ 
$\mu^{(2,3)} \varphi_1(2,3) \mu^{(2,3)} \varphi_0(2,3)^{-1}\big]$.

\end{eg}

\begin{proof} [Proof of Theorem \ref{thm:homomesy}]

Continuing with the assumption that $\bfr \geq \bfs$, we start with the case $d < \bfs \leq \bfr$ and consider iterations of birational rowmotion applied to the file $\{(\bfr-c,d-c)\}_{c=0}^d$, i.e., to 
$\{(\bfr,d), (\bfr-1,d-1), \dots, (\bfr-d,0)\}$.  From Theorem \ref{thm:main} (a), we obtain the following values for 
$\rho_B^{k+1}(\bfr-c,d-c)$:
$$\frac{ \varphi_{k}(\bfr-c-k,d-c-k)}{\varphi_{k+1}(\bfr-c-k,d-c-k)} \mathrm{~~~for~~~} 0 \leq k \leq d-c,$$
$$\mu^{(k+c-d,0)}\bigg[\frac{ \varphi_{d-c}(\bfr-d,0)}{\varphi_{d-c+1}(\bfr-d,0)}\bigg] \mathrm{~~~for~~~} d-c \leq k \leq \bfr-c,$$
$$\mu^{(k+c-d,k+c-\bfr)}\bigg[\frac{ \varphi_{d+\bfr-k-2c}(k+c-d,k+c-\bfr)}{\varphi_{d+\bfr-k-2c+1}(k+c-d,k+c-\bfr)}\bigg] \mathrm{~~~for~~~} \bfr-c \leq k \leq \bfr+d-2c.$$
And we continue using Theorem \ref{thm:main} (b) to obtain further values for $\rho_B^{k+1}(\bfr-c,d-c)$:
$$\frac{ \varphi_{k+2c-\bfr-d}(\bfr+d+1-k-c,\bfr+\bfs+1-k-c)}
{\varphi_{k+2c-\bfr-d-1}(\bfr+d+1-k-c,\bfr+\bfs+1-k-c)} \mathrm{~~~for~~~} \bfr+d+1-2c \leq k \leq \bfr+d+1-c,$$
$$\mu^{(0,k+c-\bfr-d-1)}\bigg[\frac{ \varphi_{c+1}(0,\bfs-d)}{\varphi_{c}(0,\bfs-d)}\bigg] \mathrm{~~~for~~~} \bfr+d+1-c \leq k \leq \bfr+\bfs+1-c,$$
{\scriptsize $$\mu^{(k+c-\bfr-\bfs-1,k+c-\bfr-d-1)}\bigg[\frac{ \varphi_{\bfr+\bfs+2-k}(k+c-\bfr-\bfs-1, k+c-\bfr-d-1)}
{\varphi_{\bfr+\bfs+1-k}(k+c-\bfr-\bfs-1, k+c-\bfr-d-1)}\bigg] \mathrm{~~~for~~~} \bfr+\bfs+1-c \leq k \leq \bfr+\bfs+1.$$}
Multiplying together these values over all elements in this file and for $0 \leq k \leq \bfr+\bfs+1$, many of these numerators and denominators cancel as we saw in Example \ref{eg:fileTable}.  In particular, generically, the numerator of $\rho_B^{k+1}(\bfr-c,d-c)$ cancels 
with the denominator of $\rho_B^{k}(\bfr-c-1,d-c-1)$.  After these cancellations, we are left with the product of the following contributions:

\begin{equation}\label{eq:F1}
\bigg(\prod_{c=0}^{d-1} \varphi_{0}(\bfr-c,d-c) \bigg) 
\bigg(\prod_{c=0}^{d-1} \mu^{(\bfr-c,d-c)}\bigg[\varphi_{0}(\bfr-c,d-c)\bigg] \bigg)
\bigg(\prod_{k=0}^{\bfr-d} \mu^{(k,0)}\bigg[\varphi_0(\bfr-d,0)\bigg] \bigg),
\end{equation} 
\begin{equation}\label{eq:F2}
\bigg(\prod_{c=1}^{d} \varphi_{0}(c,\bfs-d+c) \bigg)^{-1}
\bigg(\prod_{c=1}^{d} \mu^{(c,\bfs-d+c)}\bigg[\varphi_{0}(c,\bfs-d+c))\bigg] \bigg)^{-1}
\bigg(\prod_{j=0}^{\bfs-d} \mu^{(0,j)}\bigg[\varphi_0(0,\bfs-d)\bigg] \bigg)^{-1},
\end{equation}
{\scriptsize \begin{equation}\label{eq:F3}
\bigg(\prod_{k=0}^{d} \varphi_{k+1}(\bfr-k,d-k) \bigg)^{-1} 
\bigg(\prod_{k=d+1}^{\bfr} \mu^{(k-d,0)}\bigg[\varphi_{d+1}(\bfr-d,0)\bigg] \bigg)^{-1}
\bigg(\prod_{k=\bfr+1}^{\bfr+d} \mu^{(k-d,k-\bfr)}\bigg[\varphi_{\bfr+d+1-k}(k-d,k-\bfr)\bigg] \bigg)^{-1}, 
\end{equation} }
\begin{equation}\label{eq:F4}
\bigg(\prod_{k=\bfr+1-d}^{\bfr+1}\varphi_{k+d-\bfr}(\bfr+1-k,\bfr+\bfs+1-k-d) \bigg)
\bigg(\prod_{k=\bfr+2}^{\bfr+\bfs+1-d} \mu^{(0,k-\bfr-1)}\bigg[\varphi_{d+1}(0,\bfs-d)\bigg] \bigg)\times
\end{equation}
\begin{equation}\label{eq:F5}
\hspace{7em}\bigg(\prod_{k=\bfr+\bfs+2-d}^{\bfr+\bfs+1} \mu^{(k+d-\bfr-\bfs-1,k-\bfr-1)}\bigg[\varphi_{\bfr+\bfs+2-k}(k+d-\bfr-\bfs-1,k-\bfr-1)\bigg] \bigg).
\end{equation}
\begin{eg}\label{eg:fileFive}

Rearranging the leftover terms at the end of Example~\ref{eg:fileTable} to match
Equations~\eqref{eq:F1}--\eqref{eq:F5} results in:    
\[\varphi_0(4,2) \varphi_0(3,1)  \cdot  \mu^{(4,2)}\varphi_0(4,2) \mu^{(3,1)} \varphi_0(3,1) \cdot
 \varphi_0(2,0) \mu^{(1,0)} \varphi_0(2,0) \mu^{(2,0)} \varphi_0(2,0) \qquad \; (\ref{eq:F1})
\]
\[
  \varphi_0(1,2)^{-1} \varphi_0(2,3)^{-1} \cdot   
  \mu^{(1,2)} \varphi_0(1,2)^{-1} \mu^{(2,3)} \varphi_0(2,3)^{-1}  \cdot
 \varphi_0(0,1)^{-1} \mu^{(0,1)} \varphi_0(0,1)^{-1} \qquad (\ref{eq:F2})
\]
{\scriptsize \[
 \varphi_1(4,2)^{-1}  \varphi_2(3,1)^{-1} \varphi_3(2,0)^{-1} \cdot \mu^{(1,0)}
 \varphi_3(2,0)^{-1}  \mu^{(2,0)} \varphi_3(2,0)^{-1} \cdot \mu^{(3,1)} \varphi_2(3,1)^{-1}
 \mu^{(4,2)} \varphi_1(4,2)^{-1}  \qquad \quad\ (\ref{eq:F3})
\] }
\vspace{-0.75em}
\[
\varphi_1(2,3) \varphi_2(1,2) \varphi_3(0,1) \cdot \mu^{(0,1)} \varphi_3(0,1) \cdot  \mu^{(1,2)} \varphi_2(1,2) 
\mu^{(2,3)} \varphi_1(2,3). \qquad \qquad \qquad \qquad (\text{\ref{eq:F4}--\ref{eq:F5}})
\]

\end{eg}

Here the first line (\ref{eq:F1}) comes from the numerators for the $k=0$
case, as $c=0,1,\dots, d-1$, followed by the case where $k= d + \bfr - 2c$ using the
same range for $c$.  The third continued product of \eqref{eq:F1} corresponds to the $c=d$ case
while $k$ ranges over $d-c=0,1,\dots, \bfr-d = \bfr-c$. 
This captures all
\emph{numerators} of the form $\mu^{(*,*)}\varphi_{0}(*,*)$.

The second line (\ref{eq:F2}) starts with two products corresponding to the denominators in
the $k=d + \bfr+1-2c$ and the $k=\bfr + \bfs + 1$ cases, as $c=1,2,\dots,
d$.  The third continued product of \eqref{eq:F2} corresponds to the denominator in the $c=0$ case
as $k$ ranges from 
$\bfr+d+1-c = \bfr+d+1, \bfr+d+2,\dots, \bfr+\bfs+1 = \bfr+\bfs+1 - c$ (letting $j=k-\bfr-d-1$).
This captures all
\emph{denominators} of the form $\mu^{(*,*)}\varphi_{0}(*,*)$.

The third line (\ref{eq:F3}) comes from the $c=0$ case as $k=0,1,\dots, \bfr+d$, 
capturing all
\emph{denominators} of the form $\mu^{(*,*)}\varphi_{\ell }(*,*)$, for $\ell > 0$, leftover after the cancellation. 
The fourth and fifth lines (\ref{eq:F4}-\ref{eq:F5}) come from the $c=d$ case as
$k=\bfr+1-d, \bfr+2-d, \dots, \bfr+\bfs+1$, 
capturing all \emph{numerators} of the form $\mu^{(*,*)}\varphi_{\ell }(*,*)$, for $\ell > 0$, leftover after the cancellation.

We will show that the product over these five lines of contributions collapse to the value
of $1$.  We begin with Equation~\eqref{eq:F3}: its value is
identically equal to $1$ because of the families of NILPs that are involved in these
products.  In particular, the lattice path formula for $\varphi_{k+1}(\bfr-k,d-k)$
involves the points 
\[
\{s_1,s_2,\dots, s_{k+1}\} = \{(\bfr,d-k), (\bfr-1,d-k+1),\dots , (\bfr-k, d)\}
\]
and  
\[
\{t_1,t_2,\dots, t_{k+1}\} = \{(\bfr,s-k), (\bfr-1,s-k+1),\dots ,
(\bfr-k, s)\}. 
\]
Hence $\varphi_{k+1}(\bfr-k,d-k)$, as $k=0,1,\dots, d$,
corresponds to a single $(k+1)$-family of NILPs covering all elements of the rank-selected
poset $\lpreg_{(r-k,d-k)}^{k+1}$ with no elements in the complement.  Replacing $k$ with the
value $d$ or $(\bfr+d-k)$, respectively, yields analogous NILPs and we also obtain  
$\mu^{(k-d,0)}\bigg[\varphi_{d+1}(\bfr-d,0)\bigg] =
\mu^{(k-d,k-\bfr)}\bigg[\varphi_{\bfr+d+1-k}(k-d,k-\bfr)\bigg]  = 1$ for
$k=d+1,d+2,\dots, \bfr$ and $k=\bfr+1,\bfr+2,\dots, d+\bfr$,
respectively. 

Analogously, Equations~\eqref{eq:F4}--\eqref{eq:F5} are identically equal
to $1$ by the same argument, but after applying the antipodal map sending $(\bfr-k,
d-k)$ to $(k, \bfs-d+k)$ and then replacing $k$ with $(r+1-k)$.  

To finish the proof, it suffices to verify that Equations~\eqref{eq:F1} and \eqref{eq:F2} 
cancel each other out.  
Key to this is the simple form of $\varphi_{0}$ as shown in
\eqref{eq:phi0}.

We begin by noting that $\ds \prod_{c=0}^{d-1} \varphi_0(\bfr-c,d-c)$ simplifies to the product
$\ds \prod_{(i,j)\in \bigvee_{(\bfr-d+1,1)}} A_{i,j}^{\min(i-\bfr+d,j)}.$
(See the entries highlighted in \cblu{blueish} tones on the top of the left-hand side of Example \ref{eg:trip-prods}.)

Similarly, after applying the antipodal map and letting $\bigwedge_{(i,j)}:
=\{(a,b)\in P: (a,b)\leq (i,j)\}$ denote the \emph{principal order ideal} based at $(i,j)$, we obtain 
$\ds \bigg(\prod_{c=0}^{d-1} \mu^{(\bfr-c,d-c)}\bigg[\varphi_{0}(\bfr-c,d-c)\bigg] \bigg) 
= \prod_{(i,j)\in \bigwedge_{(d-1,\bfs-1)}} A_{i,j}^{\min(d-i,\bfs-j)}$
(highlighted in \cgrn{greenish} tones on the bottom of the left-hand side of Example \ref{eg:trip-prods}.).

Lastly, $\ds \bigg(\prod_{k=0}^{\bfr-d}
\mu^{(k,0)}\bigg[\varphi_0(\bfr-d,0)\bigg] \bigg) = \prod_{i=0}^{\bfr}
\prod_{j=0}^{\bfs} A_{i,j}^{\min(i+1,\bfr+1-i,d+1)}$ (as highlighted in
\cred{redish} tones on the left-hand side of Example \ref{eg:trip-prods}).  

Multiplying these three contributions together, Equation~\eqref{eq:F1} equals 
$$ 
\bigg(\prod_{(i,j)\in \bigvee_{(\bfr-d+1,1)}} A_{i,j}^{\min(i-\bfr+d,j)}\bigg)
\bigg(\prod_{(i,j)\in \bigwedge_{(d-1,\bfs-1)}} A_{i,j}^{\min(d-i,\bfs-j)}\bigg)
\bigg(\prod_{i=0}^{\bfr} \prod_{j=0}^{\bfs} A_{i,j}^{\min(i+1,\bfr+1-i,d+1)}\bigg)$$
$\ds = \prod_{i=0}^{\bfr} \prod_{j=0}^{\bfs} A_{i,j}^{\min(\bfr+1-i+j,\bfs+1+i-j,d+1)}$
where these exponents depend only on the file of the associated element, and behave palindromically about the center of the poset.

By a similar analysis, Equation~\eqref{eq:F2} 
equals $\ds \bigg(\prod_{i=0}^{\bfr} \prod_{j=0}^{\bfs} A_{i,j}^{\min(\bfr+1-i+j,\bfs+1+i-j,d+1)}\bigg)^{-1}$, 
noting that the product is built up by 
negatively sloping contributions, instead of 
positively sloping
ones, in this
case (as highlighted by the color-coding on the right-hand side of Example \ref{eg:trip-prods}).

The argument above finishes the proof of Corollary \ref{thm:homomesy} in
the first case where the top element of the file is $(\bfr,d)$ and $d< s \leq
r$.   
The second case, again with $d < s \leq r$ but where the top element of the file is instead $(d,\bfs),$ 
follows from the first case by the symmetry that replaces
$(i,j)$ with $(j,i)$. 

Finally, the third case, again with the top element of the file $(d, \bfs)$
but where $s\leq d \leq r$, follows analogously 
by using a different pattern for the values obtained by iterating birational rowmotion: 

$$\frac{ \varphi_{k}(d-c-k,\bfs-c-k)}{\varphi_{k+1}(d-c-k,\bfs-c-k)}
\mathrm{~~~for~~~} 0 \leq k \leq \bfs-c,$$ 
$$\mu^{(k+c-\bfs,0)}\bigg[\frac{ \varphi_{\bfs-c}(d-\bfs,0)}{\varphi_{\bfs-c+1}(d-\bfs,0)}\bigg] \mathrm{~~~for~~~} \bfs-c \leq k \leq d-c,$$
$$\mu^{(k+c-\bfs,k+c-d)}\bigg[\frac{ \varphi_{\bfs+d-k-2c}(k+c-\bfs,k+c-d)}{\varphi_{\bfs+d-k-2c+1}(k+c-\bfs,k+c-d)}\bigg] \mathrm{~~~for~~~} d-c \leq k \leq d+\bfs-2c.$$
$$\frac{ \varphi_{k - d - \bfs + 2c}(\bfr+\bfs+1-c-k,d+\bfs+1-c-k)}{\varphi_{k - d - \bfs + 2c - 1}(\bfr+\bfs+1-c-k,d+\bfs+1-c-k)} \mathrm{~~~for~~~} d+\bfs+1-2c \leq k \leq d+\bfs+1-c,$$
$$\mu^{(k - \bfs - 1 + c - d,0)}\bigg[\frac{ \varphi_{c+1}(\bfr-d,0)}{\varphi_{c}(\bfr-d,0)}\bigg] \mathrm{~~~for~~~} d+\bfs+1-c \leq k \leq \bfr+\bfs+1-c,$$
{\scriptsize $$\mu^{(k + c - \bfs - 1 - d,k +c - \bfr - \bfs - 1)}\bigg[\frac{ \varphi_{\bfr + \bfs + 2 - k}(k + c - \bfs - 1 - d,k + c - \bfr - \bfs - 1)}{\varphi_{\bfr + \bfs + 1 - k}(k+c - \bfs - 1 - d,k + c - \bfr - \bfs - 1)}\bigg] \mathrm{~~~for~~~} \bfr+\bfs+1-c \leq k \leq \bfr+\bfs+1.$$}
This third case
includes the possibility of the middle-most file if $\bfr+\bfs = 2d$ where the antipodal map sends elements from the top half of the file to the bottom half of the same file.  Either way, the same analysis utilizing cancellations of numerators and denominators applies. 
\end{proof}

\begin{eg}\label{eg:trip-prods}

Let $(r,s) = (4,3)$, and $d=2$, with corresponding file $F = \{(4,2), (3,1), (2,0)
\}$.  The left-hand side of Figure~\ref{fig:dc} shows the contributions corresponding to
Equation~\eqref{eq:F1},  while the right-hand side shows the contributions corresponding to the reciprocal of Equation~\eqref{eq:F2}.  
For example, in the left picture, the 6-element order filter at $(3,1)$ represents $\varphi_{0}(3,1)$,
while the 6-element interval $[(0,0), (1,2)]$ represents $\mu^{(3,1)}\varphi_{0}(3,1)$.
Either way, 
the full product can also be expressed as a product built up file-by-file as  $$(A_{40})^1
(A_{41}A_{30})^2 (A_{42}A_{31}A_{20}\cdot A_{40}A_{32}A_{21}A_{10}\cdot A_{33}
A_{22}A_{11}A_{00}\cdot A_{23}A_{12}A_{01})^3 (A_{13}A_{02})^2(A_{03})^1.$$

\begin{figure}
\caption{Illustrating the double-counting argument as in Example \ref{eg:trip-prods}.}
\label{fig:dc}

\medskip 

\includegraphics{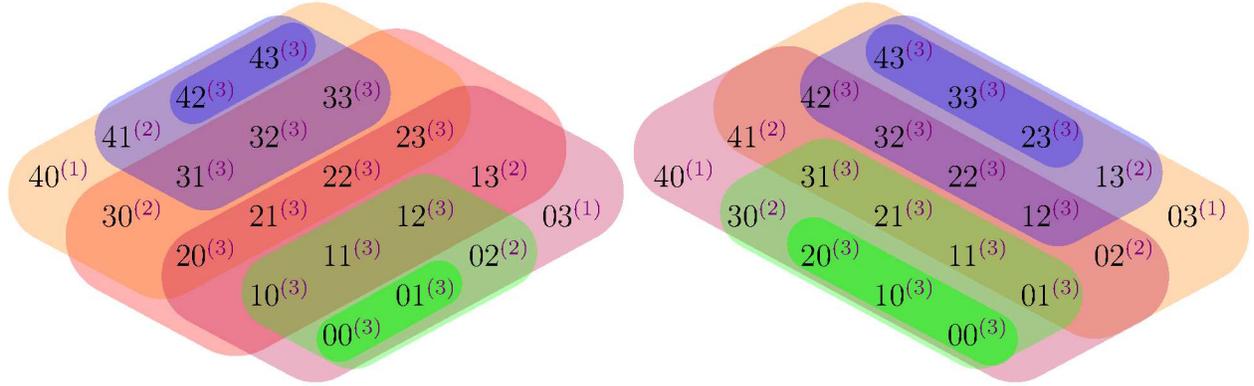}

\end{figure}

\end{eg}

\section{Connections to other works and future directions}\label{sec:connect}

We have noticed the following connections to our work and some open problems for further
exploration.

We went through a large number of candidate bijections, none of which worked, before finding
the colorful combinatorial bijections (in the proofs of
Lemmas~\ref{lem:bounce} and \ref{lem:genBounce}) at the heart of our proof of the main result.
Afterwards we found that such ``bounce-path'' bijections are already in the literature,
particularly in the work of Fulmek and Kleber~\cite{FuKl01}, which uses precisely this kind of argument
(called therein ``changing trails'') to prove Schur function identities, Dodgson's
condensation, and (unsurprisingly) Pl\"ucker relations. (This work in turn was
preceded by Goulden's earlier work \cite{Gou88} using ``$(w,z)$-alternating walks'' to
bijectively prove quadratic identites for skew-symmetric Schur functions.)   In a later
paper~\cite{Fu09}, Fulmek gives a 
number of examples to show the wide applicability of this method.   
In particular, while a Gessel-Viennot type argument would allow us to turn $\varphi_k(i,j)$ into a determinantal expression, given the specific form of identity (\ref{bounce-identity}), we did 
not see an algebraic way to prove Lemma \ref{lem:bounce} as a direct application of Dodgson condensation or the Desnanot-Jacobi identity, and instead finding the combinatorial bijection as above.

In providing a bijective proof that birational RSK satisfies the octahedron recurrence,
Farber, Hopkins, and Trungsiriwat, define a similar-looking bijection in their context of
``interlacing networks''.  They discuss the relationship between the ``changing trails'' of
\cite{FuKl01} and their ``$\tau$-involution''~\cite[p.~366]{FHT15}, pointing out that there
are significant differences as well.  It would be interesting
to gain a clearer understanding of the relationship between birational rowmotion and
birational RSK.  

Galashin and Pylyavskyy have introduced a broad generalization of birational rowmotion, called
``$R$-systems'', which are discrete dynamical systems on labelings of a 
of strongly connected directed graph.  
Given a finite poset $P$, construct a digraph $\Gamma$ by 
(a) turning each covering relation of 
$x\iscovered y$ of $\widehat{P}$ into a directed edge $y \to x$ in $\Gamma$, (b) identifying the
elements $\widehat{0}$ and $\widehat{1}$ as a single vertex $s$.  (See Remark 2.7 of
\cite{GaPy17} although their convention is the opposite of ours.)  
Since the values of an $R$-system are defined projectively, we fix the value at the vertex
$s$ to be 1 to recover our Definition~~\ref{def:bitoggle} applied to the dual
of $P$.  It is an interesting question to understand
how their formula in terms of arborescences~\cite[\S 2]{GaPy17} in the special case of a rectangular poset
compares to the $k=1$ case of our formula in Theorem~\ref{thm:main}.

Information about the relationship between birational rowmotion and the $Y$-systems of
Zamolodchikov periodicity can be found in \cite[\S 4.4]{Rob16}, the introduction of
\cite{GaPy17}, and in \cite{Volk06}.  Unpublished work of Glick and Grinberg shows that birational
rowmotion formulae are ratios of $T$-variables, while $Y$-variables are ratios of the
birational rowmotion formulae.   See \cite{Speyer}, \cite{Henriques}, or \cite{DiFK} for combinatorial formulas for the 
$T$-variables for the $A_m \times A_n$ case, also known as solutions to the octahedron recurrence.

Goncharov and Shen discuss a detropicalization of the Sch\"utzenberger involution of
Gelfand-Tsetlin patterns~\cite[\S 9.3]{GoSh16}.  They express this map $R_{a,b,c}$ as a
ratio of determinants that transforms by the same recurrence as birational rowmotion on a
rectangle. (See Equations (266) and (269).)  Another question for future research
is the relationship between our formula for iterated birational rowmotion and the role of
$R_{a,b,c}$ in~\cite{GoSh16}.  Related work also appears in work of Frieden~\cite[\S
4]{Fr17}, which has
analogous formulas for detropicalized promotion. These also can be related to
Gelfand-Tsetlin patterns, written as a ratio of determinants, and interpreted
in terms of planar networks. 

A natural open question is to find similar formulae in terms of NILPs for other situations
where rowmotion or birational rowmotion has nice periodicity.  These include several
triangular shapes obtained by cutting $[0,r]\times [0,x]$ in half vertically or
horizontally, or both~\cite[\S 9--11]{GrRo15}, as well as other types of root and minuscule
posets~\cite[\S 13]{GrRo15}.  Perhaps these would also allow one to prove birational
homomesies for these posets, analogous to Corollary~\ref{cor:reciprocity} and
Theorem~\ref{thm:homomesy}.

\end{document}